%% file: 202308Dijk.tex
\documentclass[11pt]{amsart}

\usepackage{graphicx}
\usepackage{framed}
\usepackage{ulem}
\usepackage{pgf,tikz,pgfplots}
\usepackage{mathrsfs}
\usepackage{mathtools}
\usepackage{mathpple}
\usepackage{graphicx}
\usepackage[all,cmtip]{xy}
\usepackage{amsmath}
\usepackage{tikz}
\usepackage{tikz-cd}
\usepackage{mathdots}
\usepackage{yhmath}
\usepackage{cancel}
\usepackage{color}
\usepackage{siunitx}
\usepackage{array}
\usepackage{multirow}
\usepackage{amssymb}
\usepackage{gensymb}
\usepackage{tabularx}
\usepackage{booktabs}
\usepackage{makecell}
\usetikzlibrary{fadings}
\usetikzlibrary{patterns}
\usetikzlibrary{shapes}

\usetikzlibrary{matrix,arrows,decorations.pathmorphing}
\usepackage{enumerate}
\usepackage{extarrows}
\usepackage{appendix}
\usepackage{float}
\usepackage{circledsteps}
\usepackage[bookmarks=true, bookmarksopen=true,%
bookmarksdepth=3,bookmarksopenlevel=2,%
colorlinks=true,%
linkcolor=blue,%
citecolor=blue,%
filecolor=blue,%
menucolor=blue,%
urlcolor=blue]{hyperref}

\newcommand{\op}{\operatorname}
\newcommand{\C}{\mathbb{C}}
\newcommand{\N}{\mathbb{N}}

\newcommand{\Z}{\mathbb{Z}}

\newcommand{\im}{\op{Im}}

\newcommand{\p}{\partial}

\DeclareMathOperator{\Aut}{Aut}
\DeclareMathOperator{\End}{End}

\DeclareMathOperator{\Res}{Res}
\DeclareMathOperator{\vol}{vol}

\DeclareMathOperator{\LHS}{LHS}
\DeclareMathOperator{\RHS}{RHS}
\DeclareMathOperator{\Span}{Span}

\DeclareMathOperator{\Id}{Id}

\renewcommand{\L}{\mathcal L}

\theoremstyle{plain}
\newtheorem{thm}{Theorem}[section]
\newtheorem{thm-defn}{Theorem/Definition}[section]
\newtheorem{lem}[thm]{Lemma}
\newtheorem{lem-defn}[thm]{Lemma/Definition}

\newtheorem{prop}[thm]{Proposition}

\newtheorem{cor}[thm]{Corollary}

\theoremstyle{definition}
\newtheorem{defn}[thm]{Definition}
\newtheorem{eg}[thm]{Example}

\theoremstyle{remark}
\newtheorem{rmk}[thm]{Remark}

\allowdisplaybreaks[4]  

\begin{document}
\title[Contact Term Algebras and Dijkgraaf's Master Equation]{Contact Term Algebras and Dijkgraaf's Master Equation}%
\author{Zhengping Gui, Si Li and Xinxing Tang}%

\address{Z. Gui: Shanghai Institute for Mathematics
and Interdisciplinary Sciences, Shanghai, China}
\email{gzp16@tsinghua.org.cn}

\address{S. Li: Yau Mathematical Sciences Center, Tsinghua University, Beijing, China}
\email{sili@mail.tsinghua.edu.cn}

\address{X. Tang: Beijing Institute of Mathematical Sciences and Applications, Beijing, China}
\email{tangxinxing@bimsa.cn}


\begin{abstract}
This paper is devoted to study integrable deformations of chiral conformal field theories on elliptic curves from the viewpoint of contact algebra.  We introduce the relevant integrable condition within the framework of conformal vertex algebra, and derive the contact term relations among certain local operators. We investigate three versions of genus one  partition functions and derive the contact equations. This leads to a rigorous formulation of Dijkgraaf's master equation \cite{Dijk1996master} for chiral deformations.
\end{abstract}
\maketitle
\tableofcontents
\def\Xint#1{\mathchoice
{\XXint\displaystyle\textstyle{#1}}%
{\XXint\textstyle\scriptstyle{#1}}%
{\XXint\scriptstyle\scriptscriptstyle{#1}}%
{\XXint\scriptscriptstyle\scriptscriptstyle{#1}}%
\!\int}
\def\XXint#1#2#3{{\setbox0=\hbox{$#1{#2#3}{\int}$}
\vcenter{\hbox{$#2#3$}}\kern-.5\wd0}}
\def\ddashint{\Xint=}
\def\dashint{\Xint-}

\section{Introduction}

\subsection{Physical background}

\subsubsection{Contact terms and 2d topological gravity}

Now it is well known that the genus $g$ $n$-point physical amplitudes $\langle\tau_{d_1}\cdots\tau_{d_n}\rangle_g$ of 2d topological gravity are defined as the stable intersection numbers on the moduli space of curves. Physically, Verlinde-Verlinde \cite{VV1991} give a convenient field theoretical formulation of 2d gravity as follows:
\[ \langle\tau_{d_1}\cdots \tau_{d_n}\rangle_g=\int_{\mathcal{M}_{g}}\int_{\Sigma}d^2z_n\cdots\int_{\Sigma}d^2z_1\langle\tau_{d_1}(z_1)\cdots\tau_{d_n}(z_n)\rangle_g.\]

Strictly speaking, the integration region should be compactified by adding the nodal curves and ``coincident" punctures. They show that the main contribution to the amplitude comes from contact interactions, i.e., when one operator coincides with another operator or a node. When $\tau_m$ is integrated over the Riemann surface $\Sigma$, it will pick up contact terms from the infinitesimal neighborhood of other operators. Conservations of ghost number and of curvature imply that
$$\int_{D_{\epsilon}}\tau_m|\tau_{n}\rangle=A_m^n|\tau_{n+m-1}\rangle,\quad\text{for some constants }A_m^n.$$

The consistency condition 
\begin{equation*}\label{local}
\int_{D_{\epsilon}}\tau_n\int_{D_{\epsilon}}\tau_m|\tau_k\rangle=\int_{D_{\epsilon}}\tau_m\int_{D_{\epsilon}}\tau_n|\tau_k\rangle
\end{equation*}
implies that the constants $\{A_m^n\}$ satisfy
\begin{equation}\label{contactA}
A_n^{k+m-1}A_m^k-A_m^{k+n-1}A_n^k=(A_n^m-A_m^n)A_{m+n-1}^k.
\end{equation}

Furthermore, they show that the contact term algebra is isomorphic to (half of) the Virasoro algebra and derive the Virasoro constraints for these complete amplitudes using the symmetry property
$$\langle\tau_l\tau_k\prod_{d_i}\tau_{d_i}\rangle_g=\langle\tau_k\tau_l\prod_{d_i}\tau_{d_i}\rangle_g.$$

Alternatively, by solving the descendant Ward identities, Becchi-Imbimbo \cite{CechDeRham} show that the contact terms are the contributions of the descendant local forms to the globally defined integral.

\subsubsection{Contact terms and chiral deformations of CFT on the torus}

In \cite{Dijk1996master}, Dijkgraaf studies the chiral deformations of conformal field theories on the torus, where the action with chiral interactions is of the form
$$S=S_0+\int_{E_{\tau}}\frac{d^2z}{2\pi\cdot\im\tau}t^i\Phi_i(z)$$
where
\begin{itemize}
  \item $S_0$ indicates the action of the undeformed conformal field theory,
  \item $\{\Phi_i(z)\}$ is a basis of some subalgebra satisfying a certain condition.
\end{itemize}

He shows that the deformed Lagrangian $S$ is written in terms of a deformed Hamiltonian $H$ 
$$H=H_0+\oint_Adzs^i\Phi_i(z),$$
at the expense of introducing the contact terms by requiring
$$\left\langle\exp\int_{E_{\tau}} t^i\Phi_i\right\rangle=\left\langle\exp \oint_A s^i\Phi_i\right\rangle.$$

This generalizes the model considered in  \cite{Douglas}. Along this line, a chiral deformation theory of gauged CFT is established via the BV quantization formalism in \cite{S-VOA}. 

To figure out the relation between the coupling constants $s^i$ and $t^i$, Dijkgraaf considers the mixed generating functions
\begin{equation*}\label{partZ}
Z(s,t)=\left\langle\exp\bigg(\oint_As^i\Phi_i+\int_{E_{\tau}}t^i\Phi_i\bigg)\right\rangle
\end{equation*}
and derives a dynamical equation (see Theorem \ref{DijkgraafMaster}):
\begin{equation}\label{Dijk1}
\frac{\partial Z}{\partial t^i}=\left[\frac{\partial}{\partial s^i}+\frac{1}{2\sqrt{-1}\im\tau}\left(L_i^{(s)}+L_i^{(t)}\right)\right]Z.
\end{equation}
Here $L_i^{(s)}$ and $L_i^{(t)}$ are certain 1st order differential operators. Equation \eqref{Dijk1} will be called Dijkgraaf's master equation and play a central role in this paper.

\subsubsection{Contact terms and other related topics}
The ideas of contact terms also appear in other related theories. For example, the study of contact terms in topological models coupled to topological gravity has led to the flat structure \cite{LosevDes} \cite{Losevflat} \cite{LosevPolyubin} as well as the discovery of the holomorphic anomaly \cite{BCOV} \cite{ContactHAE}.
In particular, in the study of Landau-Ginzburg model, with superpotentials $W\in\C[x]$ \cite{LosevDes} or $W=-\frac{1}{x}$ \cite{Ghoshal1995}, the corresponding (multi-)contact terms are given by explicit formulas.

In the study of the $u$-plane integrals \cite{Marcos1999} \cite{Takasaki1999}, there is also an interpretation of contact terms from the point of view of integrable hierarchies and their Whitham deformations. The multi-time tau function is comprised of an exponential factor $e^{\frac{1}{2}\sum_{m,n\geq1}q_{mn}t_mt_n}$ and a theta function. The time variables $t_n$ play the role of physical coupling constants of two-observables $I_n(B)$ carried by the exceptional divisor $B$ in the blowup. The coefficients $q_{mn}$ of the exponential part are identified as the contact terms of these two observables.

The contact terms are also explained in the chiral anomaly \cite{Kutasov} of heterotic strings and preserve the modular invariance by canceling the boundary terms.

\subsection{Main results: mathematical formulation}

In the last three subsections, we have recalled some related topics on contact terms. In concrete situations, contact terms are mainly defined along with an exact condition and computed through consistency requirements (local and global ones) or comparison with known priori results.

In this paper, we mainly focus on the study of chiral deformations of CFT on the torus. We give a mathematical definition of the mixed genus one partition function $Z(s,t)$ (see Definitions \ref{iteratedmodifiedA}, \ref{mixtype} and Equation \eqref{Zst}) as well as a rigorous proof of Dijkgraaf's master function \eqref{Dijk1}. The key constructions are summarized as follows:

\begin{enumerate}
  \item Locally, the contact term algebra will be formulated via two important operations $B$, $C$ (see Definition \ref{defBC}) on some subvertex algebra $\mathcal{A}\subset V$ (see Definition \ref{intcdt}) of a conformal vertex algebra $V$.
  \item We will deal with several modified integrals:
  \begin{itemize}
    \item The integrals $\oint_{A^n}$ are well defined on (almost-)meromorphic functions on $E_{\tau}^n$ with certain genus one properties and some integrable condition (Dijkgraaf's condition, see Definition \ref{intcdt}).
    \item The integrals $\int_{E_{\tau}^n}$ will be defined by the regularized integral $\dashint_{E_{\tau}^n}$ developed by Li-Zhou\cite{SiJie2020}.
    \item To make sense of ``$\int_{E_{\tau}^m}\oint_{A^n}$", we introduce the modified $A$-cycle integral $\widehat{\oint}_A$ (see Equation \eqref{modifiedAcycle}) and its iterated version (see Definition \ref{iteratedmodifiedA}).
  \end{itemize}
  \item Globally, by applying the Fubini Theorem of regularized integral, this contact term algebra is represented non-linearly on the amplitudes through Dijkgraaf's master equation. We provide a careful analysis on the relations among iterated regularized integrals, iterated (modified) $A$-cycle integrals and contact iterations.
\end{enumerate}

Our main results can be stated as follows:
\begin{thm}[Contact term relations, Theorem \ref{contrel}]
    For $a=a_1\otimes\cdots\otimes a_n\in\mathcal{A}^{\otimes n}$, and for distinct $i,j,k$, we have
\begin{enumerate}
\setlength{\itemsep}{2pt}
 \item $B_{i\rightarrow k} B_{j\rightarrow k}a=B_{j\rightarrow k}B_{i\rightarrow k}a.$
 \item $[C_{i\rightarrow k},B_{j\rightarrow k}]a+B_{i\rightarrow k}B_{j\rightarrow k}a=B_{j\rightarrow k}C_{i\rightarrow j}a.$
 \item $[C_{i\rightarrow k}, C_{j\rightarrow k}]a\equiv C_{j\rightarrow k}C_{i\rightarrow j}a-C_{i\rightarrow k}C_{j\rightarrow i}a~\mod(\im(T)).$
\end{enumerate}
\end{thm}

One can regard the third identity about the non-commutative relations of $C_{i\rightarrow k}$ and $C_{j\rightarrow k}$ as an operator version of the contact term relations \eqref{contactA} appearing in Verlinde-Verlinde's work.

Introducing the modified (multi-)A-cycle integrals $\widehat{A}_N$ (see Definition \ref{iteratedmodifiedA}), we prove the following key result which finally leads to the contact equation.

\begin{thm}[Theorem \ref{ModifiedAcycleElliptic}] For any $a\in \mathcal{A}^{\otimes(|M|+|N|)}$, let $[a]$ denote the corresponding function in the genus one conformal block (see Definition \ref{genus1cb}). Then $\widehat{A}_{N}[a]$  is almost-meromorphic and elliptic, also satisfies the axioms (1)-(5) of the genus one conformal block. Its composition with the regularized integral $E_m$ satisfies
  $$
  E_{m}\widehat{A}_{N}[a]=\widehat{A}_{N\cup\{m\}}[a]+\frac{1}{2\sqrt{-1}\operatorname{Im}\tau}\sum_{1\leq j\leq m+n,j\neq m}\widehat{A}_{N}[C_{m\rightarrow j}a].
  $$
\end{thm}

\begin{thm}[Theorem \ref{Contactequation} Contact equation, Theorem \ref{DijkgraafMaster}, Dijkgraaf's master equation]
Let  $v=v_{k_1}\otimes \cdots\otimes v_{k_{m+n}}$ with $v_{k_i}$ the element in the Dijkgraaf's condition. Then the mixed correlators (see Definition \ref{mixtype}) satisfy 
$$
\langle v\rangle_{m,n}=\langle v\rangle_{m-1,n+1}+\frac{1}{2\sqrt{-1}\operatorname{Im}\tau}\sum_{1\leq j\leq m+n,j\neq m}\sum_{l\in\mathbb{N}}c^l_{k_mk_j}\langle v_{k_1}\otimes \displaystyle \underset{\text{m-th}}{1}\otimes\cdots \otimes\underset{\text{j-th}} {v_l} \otimes \cdots\otimes v_{k_{m+n}}\rangle_{m,n},
$$
where the coefficients $c_{ij}^k$ are the structure constants with respect to $\{v_i\}$ (see Definition \ref{intcdt}).

Furthermore, the corresponding genus one partition function $Z=Z(s,t)$ (Equation \eqref{Zst}) satisfies the following linear differential equation
\begin{equation*}
\frac{\partial Z}{\partial t^i}=\left[\frac{\partial}{\partial s^i}+\frac{1}{2\sqrt{-1}\operatorname{Im}\tau}(L^{(s)}_i+L^{(t)}_i)\right]Z, \quad \forall ~i\in \mathbb{N},
\end{equation*}
where
$$
L^{(s)}_i=\sum_{j,k\in\mathbb{N}}c^k_{ij}s^j\frac{\partial}{\partial s^k}, \quad L^{(t)}_i=\sum_{j,k\in\mathbb{N}}c^k_{ij}t^j\frac{\partial}{\partial t^k}.
$$
\end{thm}

\subsection{Organization}

In Section \ref{Dijkcondandcontact}, we define Dijkgraaf's integrable condition via vertex operator algebra and derive some natural properties. Under this integrable condition, we introduce our key operations $B$, $C$ and derive the contact term relations that lead to the local picture of the contact term algebra. In Section \ref{Genusoneandmasterequation}, we review necessary facts about the genus one conformal block and regularized integral, and define rigorously various genus one partition functions $Z_{E_{\tau}}(t)$, $Z_A(s)$ and $Z(s,t)$. We derive the contact equation (Theorem \ref{Contactequation}) and Dijkgraaf's master equation (\eqref{Dijkmast} in Theorem \ref{DijkgraafMaster}) with proof details given in Subsection \ref{proofofpropthm}. In Appendix \ref{proofnormal}, we prove that the normal ordered Feynman graph construction for chiral bosons defines an element in the genus one conformal block.
\vskip 0.2cm

\textbf{Acknowledgment.}
The authors would like to thank Robbert Dijkgraaf and Jie Zhou for helpful communications and discussions. S. L. is partially supported by the National Key R \& D Program of China
(NO. 2020YFA0713000). X. T. is partially supported by the Youth Project No. 1254041 of BJNSF.

\section{Dijkgraaf's integrable condition and Contact term algebra}\label{Dijkcondandcontact}

In this section, we will restate Dijkgraaf's integrable condition in terms of the language of vertex algebra, and derive the contact term relations based on this condition. In Subsection \ref{cva}, we will recall necessary facts on conformal vertex algebra, for more details, please refer to \cite{Frenkel2001voa}, \cite{1996Modular}. In Subsection \ref{Dic}, we will give the definition of the integrable condition (see Definition \ref{intcdt}) which we call Dijkgraaf's condition, then examples and basic properties. In Subsection \ref{BCcontact}, we will define two important operations $B$, $C$ (see Definition \ref{defBC}) and derive the contact term relations (see Theorem \ref{contrel}).

\subsection{Conformal vertex algebra}\label{cva}

Let $(V, |0\rangle, T, Y(\cdot,z),\omega)$ be a $\mathbb{Z}_{\geq0}$-graded conformal vertex algebra of central charge $c$, that is,
\begin{enumerate}
  \item (state space) a graded vector space $V=\oplus_{n=0}^{\infty} V_n$ indexed by degree;
  \item (vacuum) a vector $|0\rangle\in V_0$;
  \item (translation operator) a degree 1 linear operator $T:V\rightarrow V$;
  \item (state-field correspondence) a linear operation (vertex operators)
  \begin{align*}
  Y(\cdot,z): &V\longrightarrow \End(V)[[z,z^{-1}]]\\
              &a\longmapsto Y(a,z)=\sum_{n\in\Z}a_{(n)}z^{-n-1};  \quad(a_{(n)}\in\End(V))
  \end{align*}
  \item (conformal vector) $\omega\in V_2$, $c\in\C$
\end{enumerate}
satisfying the following axioms:
\begin{itemize}
  \item (vacuum axiom) $Y(|0\rangle,z)=\Id_{V}$. Furthermore, for any $a\in V$, we have
         $$Y(a,z)|0\rangle\in V[[z]],\quad \lim_{z\rightarrow0}Y(a,z)|0\rangle=a;$$
  \item (translation axiom) $T|0\rangle=0$. For any $a\in V$, $Y(Ta,z)=\frac{d}{dz}Y(a,z)$;
  \item (locality axiom) All fields $Y(a,z)$, $a\in V$, are mutually local, i.e. for any $a,b\in V$, there exists $N\in\Z_+$, such that
  $$(z-w)^N[Y(a,z),Y(b,w)]=0;$$
  \item (energy-momentum tensor) $T(z):=Y(\omega,z)=\sum_{n\in\Z}L_nz^{-n-2}$, satisfies
  \begin{itemize}
  \item $L_{-1}=T$, $L_0 a=na$, for $a\in V_n$;
  \item $\{L_n\}_{n\in\Z}$ span the Virasoro algebra with central charge $c$.
\end{itemize}
\end{itemize}

There are some basic consequences that we will use later.

\begin{prop}\label{immediateconseq} Let $V$ be a $\Z_{\geq0}$-graded vertex algebra.
\begin{enumerate}
  \item For $a\in V_m$, $a_{(n)}\in\End(V)$ has degree $m-n-1$, i.e.
  $$a_{(n)}V_k\subset V_{k+m-n-1};$$
  \item $(Ta)_{(n)}=-na_{(n-1)}$, in particular, $(Ta)_{(0)}=0$;\\
$T(a_{(n)}b)+na_{(n-1)}b=a_{(n)}Tb$, in particular, $T(a_{(0)}b)=a_{(0)}Tb$.
\end{enumerate}
\end{prop}

Let $a,b\in V$. Their operator product expansion (OPE) can be expanded as
$$Y(a,z)Y(b,w)=\sum_{n\in\Z}\frac{Y(a_{(n)}b,w)}{(z-w)^{n+1}},$$
where $a_{(n)}b$ can be also viewed as defining an infinite tower of binary operations. 
\begin{prop}\label{nthprod}
Let $V$ be a $\Z_{\geq0}$-graded vertex algebra.
\begin{enumerate}
  \item For homogeneous elements $a,b\in V$,
  $$a_{(n)}b=\sum_{k=0}^{\infty}(-1)^{k+n+1}\frac{1}{k!}T^k(b_{(n+k)}a);$$
  \item (Borcherds Formula) For $a,b\in V$, every $l,m,n\in \mathbb{Z}$ :
$$\sum_{j\geq0}\binom{l}{j}\left[(-1)^ja_{(m+l-j)}b_{(n+j)}-(-1)^{j+l}b_{(n+l-j)}a_{(m+j)}\right]=\sum_{k\geq0}\binom{m}{k}\left(a_{(k+l)}\cdot b\right)_{(m+n-k)}.$$
\end{enumerate}
\end{prop}

For convenience, we list some simple cases of the Borcherds Formula.
\begin{cor}\label{Borcherdsspec} (1) For $l=n=0,m=1$,
$$a_{(1)}b_{(0)}-b_{(0)}a_{(1)}=\left(a_{(0)}b\right)_{(1)}+\left(a_{(1)}b\right)_{(0)}.$$
(2) For $l=0,m=n=1$,
$$a_{(1)}b_{(1)}-b_{(1)}a_{(1)}=\left(a_{(0)}b\right)_{(2)}+\left(a_{(1)}b\right)_{(1)}.$$
(3) For $m=n=0, l=2$,
$$a_{(2)}b_{(0)}-b_{(2)}a_{(0)}-2a_{(1)}b_{(1)}+2b_{(1)}a_{(1)}+a_{(0)}b_{(2)}-b_{(0)}a_{(2)}=\left(a_{(2)}b\right)_{(0)}.$$
In particular, $\left(a_{(2)}b\right)_{(0)}$ is skew-symmetric with respect to $a$ and $b$.
\end{cor}

\begin{prop}\label{kerT} Let $V$ be a $\Z_{\geq0}$-graded conformal vertex algebra. Then
$$\ker T\subset V_0.$$
\end{prop}

\begin{proof} It follows from the standard $sl_2$ relations of $\{L_{-1},L_0,L_1\}$ and degree condition. \qedhere
\end{proof}

\subsection{Dijkgraaf's integrable condition}\label{Dic}

Following Dijkgraaf's work \cite{Dijk1996master}, we introduce the following integrable condition\footnote{It is slightly different from Dijkgraaf's original statement, where he requires that the residue of OPE $A(z)B(w)$ belongs to $\ker(T)$, see \cite{Dijk1996master} (2.8). We find it more natural to work with $\im(T)$ as in Definition \ref{intcdt}.}.

\begin{defn}\label{intcdt} Let $(V, |0\rangle, Y(\cdot,z),\omega)$ be a $\mathbb{Z}_{\geq0}$-graded conformal vertex algebra. We say a subvertex algebra $\mathcal{A}\subset V$ satisfies the Dijkgraaf's condition, if
\begin{itemize}
\item $\mathcal{A}\cap V_1=0$,
\item For any $a,b\in\mathcal{A}$,
   $$a_{(0)}b\in T\mathcal{A},$$
   or equivalently,
   $$\Res_{z=w}Y(a,z)Y(b,w) \text{ is a total derivative.}$$
\end{itemize}
\end{defn}

In the derivation of Dijkgraaf's master equation, we will furthermore require that $\mathcal{A}$ has a countable number of elements $\{v_i\}_{i\in\mathbb{N}}$, such that
\begin{itemize}
  \item $\{v_i\}_{i\in\mathbb{N}}\subset \mathcal{A}_{>0}:=\mathcal{A}\cap \oplus_{n>0}V_n$ ,
  \item there exist constants $\{c^k_{ij}\}$ in $\mathbb{C}$, such that $$v_{j\ (0)}v_i-T\sum_{k\in\mathbb{N}}c^k_{ij}v_k\in T^2\mathcal{A}.$$
\end{itemize}

Immediately, we have
\begin{lem}\label{dijkcor} Let the subvertex algebra $\mathcal{A}\subset V$ satisfy the Dijkgraaf's condition. Then
\begin{enumerate}
  \item for $a,b\in\mathcal{A}$, we have
        $$[a_{(0)},b_{(0)}]=0.$$
  \item $T(\mathcal{A}_0)=0$.
\end{enumerate}
\end{lem}

\begin{rmk} Regarding $\mathcal{A}$ as a space encoding integrable deformations. One can always choose $\omega\in\mathcal{A}$, then $T(z)=Y(\omega,z)$ gives the energy-momentum tensor deformation. Therefore, the other elements $a\in\mathcal{A}_{>2}$ give the higher order (integrable) deformations.
\end{rmk}

We will give several examples of $\mathcal{A}$ in Subsection \ref{Dijkeg}. Before the examples, let us give one more important lemma associated to $\mathcal{A}$ for our future definition.

\begin{lem}\label{T-Inverse}
  There exists a well-defined linear operator
  $$
  T^{-1}: \mathcal{A}_{(0)}\mathcal{A}\rightarrow \mathcal{A}_{>0}
  $$
  which fits into the commutative diagram. 
  $$
\begin{tikzcd}
0 \arrow[r] & \mathcal{A}_0 \arrow[r] & \mathcal{A} \arrow[r, "T"] & T(\mathcal{A})                                                    \\
            &                         &                            & \mathcal{A}_{(0)}\mathcal{A} \arrow[hookrightarrow,u, "i"'] \arrow[lu, "T^{-1}"]
\end{tikzcd}
  $$
  In particular, we have  $T(T^{-1}a)=a, a\in\mathcal{A}_{(0)}\mathcal{A}$ and  
  $$-n(T^{-1}a)_{(n-1)}=a_{(n)},\quad \text{for } n\geq1.$$
\end{lem}

\begin{proof} For $a,b\in\mathcal{A}$, assume $a_{(0)}b\neq0$. Then $\exists~c\in\mathcal{A}$ such that $a_{(0)}b=Tc$. Define
$$T^{-1}(a_{(0)}b)=\operatorname{Pr}(c),$$
here $\operatorname{Pr}:V\rightarrow \oplus_{n>0}V_n$ is the natural projection. By Lemma \ref{dijkcor} (2), we have $T(c-\operatorname{Pr}(c))=0$. Then
$$T(\operatorname{Pr}(c))=T(c)=a_{(0)}b.$$
Now we show $T^{-1}$ is well-defined, that is, it does not depend on the choice of $c$. Suppose that there exists another element $c'\in\mathcal{A}$, such that
$$T(c')=T(c)=a_{(0)}b,$$
then
$$T(\operatorname{Pr}(c')-\operatorname{Pr}(c))=0.$$
By Proposition \ref{kerT}, $\ker(T)\subset V_0$, then $\operatorname{Pr}(c')=\operatorname{Pr}(c)$. The remaining part follows from Proposition \ref{immediateconseq}.
\end{proof}

\subsubsection{Examples}\label{Dijkeg}
\begin{eg}[Vertex operator algebra from Novikov algebra]
This example is studied in \cite{BaiLiPei}.
We first recall the definition of a Novikov algebra (see \cite{BaiLiPei} and references therein). A Novikov algebra is a non-associative algebra $A$ with the product satisfying 
  \begin{align*}
    (ab)c-a(bc) & =(ba)c-b(ac) \\
    (ab)c &= (ac)b
  \end{align*}
  for $a,b,c\in \mathcal{A}$.  We also assume that $A$ is equipped with a symmetric bilinear form $\langle\cdot,\cdot\rangle$, satisfying
  \begin{equation}\label{bilinearform}
  \langle ab,c\rangle=\langle a,bc\rangle=\langle ba,c\rangle, ~\forall ~a,b,c \in A.
  \end{equation}

  Set
  $$
\widehat{A}=A\otimes \mathbb{C}[t,t^{-1}]\oplus\C\mathbf{c},  ~~\widehat{A}_+=A\otimes \mathbb{C}[t]\oplus \mathbb{C}\mathbf{c},  ~~   \widehat{A}_-=A\otimes t^{-1}\mathbb{C}[t^{-1}].
  $$
  Let $l\in \mathbb{C}$ and denote by $\mathbb{C}_l$ the one-dimensional $\widehat{A}_+$ -module with $\mathbf{c}$ acting as scalar $l$ and with $A\otimes \mathbb{C}[t]$ acting trivially. There exists a vertex algebra structure on the induced module
  $$
  V_{A}=U(\widehat{A})\otimes_{U(\widehat{A}_+)}\mathbb{C}_l,
  $$
  such that $(V_{A})_2=A$ and for any $a,b\in A$
  $$
  Y(a,z)Y(b,w)\sim \frac{\partial_wY(ba,w)}{z-w}+\frac{Y(ab,w)+Y(ba,w)}{(z-w)^2}+\frac{1}{12}\frac{\langle a,b\rangle\mathbf{c}}{(z-w)^3}.
  $$
  Then we can choose $\mathcal{A}\subset V_{A}$ to be the subvertex algebra generated by $|0\rangle$ and $A$ with
  \[a_{(0)}b=T(ba),\quad a_{(1)}b=ab+ba,\quad a_{(3)}b=\frac{1}{2}l\langle a,b\rangle|0\rangle,\quad a_{(2)}b=a_{(k\geq4)}b=0,\]
  for any $a,b\in A$. Let $\{v_i\}$ be a basis of $A$, then the coefficients $c^k_{ij}$ in Definition \ref{intcdt} are the structure constants of the Novikov algebra $A$.
\end{eg}

\begin{rmk}
A unital Novikov algebra $A$ with a symmetric bilinear form satisfying \eqref{bilinearform} amounts to a Frobenius algebra, i.e., a unital commutative and associative algebra with a non-degenerate symmetric and associative form. For a Frobenius algebra
$A$, a similar construction implies that the associated vertex algebra $L(A)$ \cite{Lam} is isomorphic to the corresponding Virasoro algebra, and $\big(L(A)_2,~a\cdot b:=a_{(1)}b\big)\cong A$ as an algebra. Thus, the structure constants $c_{ij}^k$ in the Frobenius algebra coincide with those in our definition. Such structure constants have important applications in the Poisson brackets of hydrodynamic type \cite{BN85}.  
\end{rmk}

\begin{eg}[$W_{\infty}$ algebra]\label{Walg} In terms of the free boson $\phi(z)$ with OPE 
$$\phi(z)\phi(w)\sim\log(z-w),$$
the fields of $W_{1+\infty}$ vertex operator algebra are given through
$$:e^{\phi(z)-\phi(w)}:=1+\sum_{k\geq1}\frac{(z-w)^k}{k!}W^{(k)}(w).$$

That is, for $k\geq1$,
\begin{equation}\label{Wk}
W^{(k)}(z)=\sum_{\sum_{i\geq1}ik_i=k}\frac{k!}{\prod k_i!}:\prod_i\left(\frac{1}{i!}\partial^i\phi\right)^{k_i}:.
\end{equation}

By the boson-fermion correspondence, a free boson is equivalent to a pair of fermions, that is
$$\psi=:e^{\phi}:,\quad \psi^{\dag}=:e^{-\phi}:\quad \text{and}\quad \partial\phi=:\psi\psi^{\dag}:.$$
where the OPE of the pair of fermions are given by
$$\psi(z)\psi^{\dag}(w)\sim \frac{1}{z-w}.$$
It is known that
$$:\psi(z)\psi^{\dag}(w):=\frac{1}{z-w}\left(:e^{\phi(z)-\phi(w)}:-1\right).$$

\noindent So we can also express $W^{(k)}(z)$ in terms of fermions,
$$W^{(k)}(z)=k:(\partial^{k-1}\psi)\psi^{\dag}:.$$
For simplicity, rescaling $W^{(k)}(z)$ by $\frac{1}{k}$, i.e.
$$\widetilde{W}^{(k)}(z):=\frac{1}{k}W^{(k)}(z)=:(\partial^{k-1}\psi)\psi^{\dag}:.$$

\begin{lem}\label{WOPE}$\Res_{z=w}\widetilde{W}^{(m)}(z) \widetilde{W}^{(n)}(w)$ is a total derivative, that is
\begin{align*}
\Res_{z=w}\widetilde{W}^{(m)}(z) \widetilde{W}^{(n)}(w)=~&\sum_{k=1}^{m-1}\binom{m-1}{k}(-1)^{k-1}\p^{k}\widetilde{W}^{(m+n-k-1)}(w)\\
                               =~&(m-1)\p \widetilde{W}^{(m+n-2)}(w)-\sum_{k=2}^{m-1}\binom{m-1}{k}(-\p)^{k}\widetilde{W}^{(m+n-k-1)}(w).
\end{align*}
\end{lem}

\begin{proof} Using the fermionic representation, we can compute
\begin{align*}
&\Res_{z=w}\widetilde{W}^{(m)}(z) \widetilde{W}^{(n)}(w)\\
=~&\Res_{z=w}\left(:(\partial^{m-1}\psi)\psi^{\dag}(z): ~:(\partial^{n-1}\psi)\psi^{\dag}(w):\right)\\
=~&\Res_{z=w}\left(\p_z^{m-1}\frac{1}{z-w}:\psi^{\dag}(z)\p^{n-1}\psi(w):+\p_w^{n-1}\frac{1}{z-w}:\p^{m-1}\psi(z)\psi^{\dag}(w):\right)\\
=~&(-1)^{m-2}:\p^{n-1}\psi(w)\p^{m-1}\psi^{\dag}(w):+:\p^{m+n-2}\psi(w)\psi^{\dag}(w):\\
=~&\p\left[:\p^{m+n-3}\psi\psi^{\dag}:-:\p^{m+n-4}\psi\p\psi^{\dag}:+\cdots+(-1)^{m-2}:\p^{n-1}\psi\p^{m-2}\psi^{\dag}:\right](w).
\end{align*}
Furthermore, one can prove by induction that
\begin{align*}
&:\p^{m+n-3}\psi\psi^{\dag}:-:\p^{m+n-4}\psi\p\psi^{\dag}:+:\p^{m+n-5}\psi\p^2\psi^{\dag}:-\cdots+(-1)^{m-2}:\p^{n-1}\psi\p^{m-2}\psi^{\dag}:\\
=~&\sum_{k=1}^{m-1}\binom{m-1}{k}(-\p)^{k-1}:(\p^{m+n-2-k}\psi \psi^{\dag}):\\
=~&\sum_{k=1}^{m-1}\binom{m-1}{k}(-\p)^{k-1}\widetilde{W}^{(m+n-k-1)}. \qedhere
\end{align*}
\end{proof}

Note that each field $\widetilde{W}^{(i)}(z)$ has degree $i$. The algebra generated by $\{\widetilde{W}^{(i)}(z)\}_{i\geq2}$ is the $W_{\infty}$ algebra. By the bosonic defining equation \eqref{Wk} of $W_{1+\infty}$ algebra, for each $\widetilde{W}^{(i)}(z)$, there is an element $v_i\in V$, such that
$$\widetilde{W}^{(i)}(z)=Y(v_i,z),\quad i\geq 2.$$
Then we consider the subvertex algebra $\mathcal{A}\subset V$ spanned by the vacuum state $|0\rangle$ and $\{v_i\}_{i\geq2}$. It is obvious that the $W_{\infty}$ algebra $\mathcal{A}$ satisfies \emph{Dijkgraaf's condition} and the structure constants $c_{ij}^k=(j-1)\delta_{k,i+j-2}$ in terms of $\{v_i\}$.
\end{eg}

\subsection{The operations $B$, $C$ and contact term algebra}\label{BCcontact}

Inspired by Dijkgraaf's work \cite{Dijk1996master}, we consider the following definitions.
\begin{defn}\label{defBC}
Let $\mathcal{A}\subset V$ be a vertex subalgebra satisfying Dijkgraaf's condition. We define the following two operations
$$B:\mathcal{A}^{\otimes2}\rightarrow \mathcal{A}, \quad C:\mathcal{A}^{\otimes2}\rightarrow \mathcal{A}$$
by
\begin{equation}\label{BC1}
\begin{aligned}
&B(a_i,a_j)=T^{-1}(a_{i\ (0)}a_j), \quad\text{(It is well defined by Lemma \ref{T-Inverse}.)}\\
&C(a_i,a_j)=a_{i\ (1)}a_j-B(a_i,a_j).
\end{aligned}
\end{equation}
\end{defn}

By Proposition \ref{nthprod} (1), the formulas for $n=0,1$ give
\begin{align*}
&a_{i\ (0)}a_j+a_{j\ (0)}a_i=\sum_{k\geq 1}(-1)^{k+1}\frac{T^k(a_{i\ (k)}a_j)}{k!},\\
&a_{i\ (1)}a_j=\sum_{k\geq 0}(-1)^{k+2}\frac{T^k(a_{j\ (k+1)}a_i)}{k!}.
\end{align*}
Then, one can easily derive the important relation
\begin{equation}\label{BC2}
C(a_i,a_j)-B(a_j,a_i)\in \im T.
\end{equation}
Note also that $B(a_i,a_j), C(a_i,a_j)\in\mathcal{A}_{>0}$.

Furthermore, we can extend these operations on ${\mathcal{A}}^{\otimes n}$ by
\begin{align*}
T_{i}(a_{1}\otimes\cdots\otimes a_{n}):&=a_{1}\otimes \cdots\otimes\displaystyle \underset{\text{i-th}}{Ta_i}\otimes\cdots \otimes a_{n}, \ 1\leq i\leq n,\\
B_{i\rightarrow j}(a_{1}\otimes\cdots\otimes a_{n}):&=a_{1}\otimes \displaystyle \underset{\text{i-th}}{|0\rangle}\otimes\cdots \otimes\underset{\text{j-th}} {B(a_i,a_j)} \otimes \cdots\otimes a_{n}, \ 1\leq i\neq j\leq n,\\
C_{i\rightarrow j}(a_{1}\otimes\cdots\otimes a_{n}):&=a_{1}\otimes \displaystyle \underset{\text{i-th}}{|0\rangle}\otimes\cdots \otimes\underset{\text{j-th}} {C(a_i,a_j)} \otimes \cdots\otimes a_{n}, \ 1\leq i\neq j\leq n.
\end{align*}

By definition, it is straight-forward to check that
\begin{lem}\label{trivialrel} Let $a=a_1\otimes\cdots\otimes a_n\in\mathcal{A}^{\otimes n}$, for distinct $i,j,k,l$, we have
\begin{align*}
    B_{i\rightarrow j}B_{k\rightarrow l}a= ~&B_{k\rightarrow l}B_{i\rightarrow j}a,\quad C_{i\rightarrow j}C_{k\rightarrow l}a=C_{k\rightarrow l}C_{i\rightarrow j}a;\\
    &B_{i\rightarrow j}C_{k\rightarrow l}a=C_{k\rightarrow l}B_{i\rightarrow j}a.
\end{align*}

\end{lem}

Next, we consider the remaining non-trivial relations between $B$ and $C$.

\begin{thm}[Contact term relations]\label{contrel}
For $a=a_1\otimes\cdots\otimes a_n\in\mathcal{A}^{\otimes n}$, and for distinct $i,j,k$, we have
\begin{enumerate}
\setlength{\itemsep}{2pt}
 \item $B_{i\rightarrow k} B_{j\rightarrow k}a=B_{j\rightarrow k}B_{i\rightarrow k}a.$
 \item $[C_{i\rightarrow k},B_{j\rightarrow k}]a+B_{i\rightarrow k}B_{j\rightarrow k}a=B_{j\rightarrow k}C_{i\rightarrow j}a.$
 \item $[C_{i\rightarrow k}, C_{j\rightarrow k}]a\equiv C_{j\rightarrow k}C_{i\rightarrow j}a-C_{i\rightarrow k}C_{j\rightarrow i}a~\mod(\im(T)).$
\end{enumerate}
\end{thm}

\begin{proof} (1) By definition, the $k$-th element of $B_{i\rightarrow k}B_{j\rightarrow k}a$ is
$$B(a_i,B(a_j,a_k))=T^{-1}(a_{i\ (0)}(T^{-1}(a_{j\ (0)}a_k))).$$
First, we compute $T(a_{i\ (0)}(T^{-1}(a_{j\ (0)}a_k)))$. By Proposition \ref{immediateconseq} (2), Lemma \ref{dijkcor} (1), we have
 \begin{align*}
     T(a_{i\ (0)}(T^{-1}(a_{j\ (0)}a_k)))& =a_{i\ (0)}(a_{j\ (0)}a_k)\quad\text{(Proposition \ref{immediateconseq}, $T(a_{(0)}b)=a_{(0)}Tb$.)}\\
     &=a_{j\ (0)}(a_{i\ (0)}a_k) \quad \text{(Lemma \ref{dijkcor}, $[a_{i\ (0)},a_{j\ (0)}]=0$.)}\\
     &=T(a_{j\ (0)}(T^{-1}(a_{i\ (0)}a_k))).
 \end{align*}

Since $\ker T\subset V_0$, by degree reason, we have
$$a_{i\ (0)}(T^{-1}(a_{j\ (0)}a_k))=a_{j\ (0)}(T^{-1}(a_{i\ (0)}a_k))\in\mathcal{A}_{>0}.$$
Furthermore,
$$T^{-1}(a_{i\ (0)}(T^{-1}(a_{j\ (0)}a_k)))=T^{-1}(a_{j\ (0)}(T^{-1}(a_{i\ (0)}a_k))).$$

\vskip 0.2cm
(2) By definition, the $k$-th elements of $C_{i\rightarrow k}B_{j\rightarrow k}a$ and $B_{j\rightarrow k}C_{i\rightarrow k}a$ are respectively given by
\begin{align*}
&C(a_i, B(a_j, a_k))=a_{i\ (1)}T^{-1}(a_{j\ (0)}a_k)-B(a_i,B(a_j,a_k)),\\
&B(a_j, C(a_i, a_k))=T^{-1}(a_{j\ (0)}(a_{i\ (1)}a_k))-B(a_j,B(a_i,a_k)).
\end{align*}
By Theorem \ref{contrel} (1) we just proved, the $k$-th element of $[C_{i\rightarrow k},B_{j\rightarrow k}]a$ is
\begin{equation}\label{keyone}
    a_{i\ (1)}T^{-1}(a_{j\ (0)}a_k)-T^{-1}(a_{j\ (0)}(a_{i\ (1)}a_k)).
\end{equation}
Applying $T$ to it and using Proposition \ref{immediateconseq}, Corollary \ref{Borcherdsspec}, Lemma \ref{T-Inverse}, we can compute
\begin{align*}
   & a_{j\ (0)}(a_{i\ (1)}a_k)-T(a_{i\ (1)}T^{-1}(a_{j\ (0)}a_k))  \\
  =~& a_{j\ (0)}(a_{i\ (1)}a_k)-a_{i\ (1)}(a_{j\ (0)}a_k)+a_{i\ (0)}T^{-1}(a_{j\ (0)}a_k) \quad \text{(Proposition \ref{immediateconseq} (2))}\\
  =~&-(a_{i\ (0)}a_j)_{(1)}a_k-(a_{i\ (1)}a_j)_{(0)}a_k+a_{i\ (0)}T^{-1}(a_{j\ (0)}a_k)\quad \text{(Corollary \ref{Borcherdsspec} (1))}\\
  =~&(T^{-1}(a_{i\ (0)}a_j))_{(0)}a_k-(a_{i\ (1)}a_j)_{(0)}a_k+a_{i\ (0)}T^{-1}(a_{j\ (0)}a_k)\quad \text{(Lemma \ref{T-Inverse})}\\
  =~&-(C(a_i,a_j))_{(0)}a_k+T(B(a_i, B(a_j, a_k))).
\end{align*}
Applying $T^{-1}$ to both sides, we get the desired equation.

\vskip 0.2cm
(3) By definition, the $k$-th elements of $C_{i\rightarrow k}C_{j\rightarrow k}a$ and $C_{j\rightarrow k}C_{i\rightarrow k}a$ are respectively given by
\begin{align*}
C(a_i, C(a_j, a_k))=~&a_{i\ (1)}(a_{j\ (1)}a_k)-a_{i\ (1)}T^{-1}(a_{j\ (0)}a_k)\\
                    &-T^{-1}\big(a_{i\ (0)}(a_{j\ (1)}a_k)\big)+B(a_i,B(a_j,a_k)),\\
C(a_j, C(a_i, a_k))=~&a_{j\ (1)}(a_{i\ (1)}a_k)-a_{j\ (1)}T^{-1}(a_{i\ (0)}a_k)\\
                    &-T^{-1}\big(a_{j\ (0)}(a_{i\ (1)}a_k)\big)+B(a_j,B(a_i,a_k)).
\end{align*}

By (1) (2) we just proved and Corollary \ref{Borcherdsspec}, we have
\begin{align*}
  &C(a_i, C(a_j, a_k))-C(a_j, C(a_i, a_k))\\
=~&a_{i\ (1)}(a_{j\ (1)}a_k)-a_{i\ (1)}T^{-1}(a_{j\ (0)}a_k)-T^{-1}\big(a_{i\ (0)}(a_{j\ (1)}a_k)\big)\\
&-\bigg(a_{j\ (1)}(a_{i\ (1)}a_k)-a_{j\ (1)}T^{-1}(a_{i\ (0)}a_k)-T^{-1}\big(a_{j\ (0)}(a_{i\ (1)}a_k)\big)\bigg)\\
=~&\big(a_{i\ (0)}a_j\big)_{(2)}a_k+\big(a_{i\ (1)}a_j\big)_{(1)}a_k+B(C(a_j,a_i),a_k)-B(C(a_i,a_j),a_k)\\
&\qquad\text{(By Corollary \ref{Borcherdsspec} and \eqref{keyone}, (1) (2).)}\\
=~&-2B(a_i,a_j)_{(1)}a_k+B(a_{i\ (1)}a_j,a_k)+C(a_{i\ (1)}a_j,a_k)+B(C(a_j,a_i),a_k)-B(C(a_i,a_j),a_k)\\
=~&-2B(B(a_i,a_j),a_k)-2C(B(a_i,a_j),a_k)+B(B(a_i,a_j),a_k)+B(C(a_i,a_j),a_k)\\
  &+C(B(a_i,a_j),a_k)+C(C(a_i,a_j),a_k)+B(C(a_j,a_i),a_k)-B(C(a_i,a_j),a_k)\\
=~&-B(a_i,a_j)_{(1)}a_k+C(C(a_i,a_j),a_k)-C(C(a_j,a_i),a_k)+C(a_j,a_i)_{(1)}a_k\\
\equiv~&C(C(a_i,a_j),a_k)-C(C(a_j,a_i),a_k)\quad \mod (\im(T)).
\end{align*}
The last step is obtained as follows. First, by \eqref{BC2}, there exists $b\in\mathcal{A}$ such that
\[C(a_j,a_i)_{(1)}a_k-B(a_i,a_j)_{(1)}a_k=(Tb)_{(1)}a_k.\]
Then by Proposition \ref{nthprod} (1), Proposition \ref{immediateconseq} (2) and the integrable condition, we have
\[(Tb)_{(1)}a_k\equiv a_{k~(1)}Tb\equiv a_{k~(0)}b\equiv0\quad\mod(\im(T)).\qedhere\]
\end{proof}

\begin{rmk} 
By the above relation (1), for any non-empty proper subset $I\subset\{1,\ldots,n\}$ and any $k\in\{1,\ldots,n\}\setminus I$, we can define
$$B_{I\rightarrow k}a:=\prod_{i\in I}B_{i\rightarrow k}a.$$
More generally, one can consider two disjoint sets $I$ and $M$, and a map $\phi:I\rightarrow M$, and define
$$B^{\phi}_{I\rightarrow M}:=\prod_{k\in M,\phi^{-1}(k)\neq \emptyset}B_{\phi^{-1}(k)\rightarrow k}.$$
Such a definition will appear in our study of contact iterations later. 

However for $B_{j\rightarrow k}B_{i\rightarrow j}$ and $B_{i\rightarrow k}B_{j\rightarrow i}$,  they are not equal to each other.
\end{rmk}

\begin{prop}\label{Useful}
Let $I\subset \{1,\dots,n\}-\{i,k\}$ be a non-empty subset, we have
$$
  [C_{i\rightarrow k}, B_{I\rightarrow k}]+|I|\cdot B_{I\cup\{i\}\rightarrow k}=\sum_{l\in I}B_{I\rightarrow k}C_{i\rightarrow l}.
  $$
\end{prop}
\begin{proof}
  We prove this lemma by induction on $|I|$. For $|I|=1$, this is exactly Theorem \ref{contrel} (2). For $|I|>1,j\in I$, we have
  \begin{align*}
     & [C_{i\rightarrow k}, B_{I\rightarrow k}]+|I|\cdot B_{I\cup\{i\}\rightarrow k} \\
   =~&\left([C_{i\rightarrow k}, B_{I-\{j\}\rightarrow k}]B_{j\rightarrow k}+(|I|-1)\cdot B_{I\cup\{i\}\rightarrow k}\right)+\left(B_{I-\{j\}\rightarrow k}[C_{i\rightarrow k}, B_{j\rightarrow k}]+B_{I\cup\{i\}\rightarrow k}\right)\\
   =~&\sum_{l\in I-\{j\}}B_{I-\{j\}\rightarrow k}C_{i\rightarrow l}B_{j\rightarrow k}+B_{I-\{j\}\rightarrow k}B_{j\rightarrow k}C_{i\rightarrow j}\\
    &\qquad\text{(By induction hypothesis and Theorem \ref{contrel} (2))}\\
   =~&\sum_{l\in I-\{j\}}B_{I\rightarrow k}C_{i\rightarrow l}+B_{I-\{j\}\rightarrow k}B_{j\rightarrow k}C_{i\rightarrow j} \quad\text{(Lemma \ref{trivialrel})}\\
   =~&\sum_{l\in I}B_{I\rightarrow k}C_{i\rightarrow l}.\qedhere
  \end{align*}
\end{proof}

\section{Genus one partition function and Dijkgraaf's master equation}\label{Genusoneandmasterequation}

In this section, we will discuss the mathematical definitions of various genus one partition functions via the following two-step construction.
\begin{itemize}
    \item \textbf{Step 1.} Using the conformal block (Definition \ref{genus1cb}) to get a meromorphic elliptic function on the configuration space of elliptic curves 
    $$(a_1,\ldots,a_n)\mapsto S_n((a_1,z_1),\ldots,(a_n,z_n),\tau).$$
    \item \textbf{Step 2.} Applying the various integrals on the function $S_n((a_1,z_1),\ldots,(a_n,z_n),\tau)$:
    \begin{itemize}
        \item (pure type 1). regularized integral (Subsection \ref{regularizedintegrals}, Equation \eqref{regint})
$$\dashint_{E_{\tau}^n}S_n((a_1,z_1),\ldots,(a_n,z_n),\tau)\vol;$$
        \item (pure type 2). interated $A$-cycle integral
$$\oint_{A_n}\cdots\oint_{A_1}S_n((a_1,z_1),\ldots,(a_n,z_n),\tau)~dz_1\cdots dz_n;$$
        \item (mixed type). regularized integral and modified iterated $A$-cycle integral (Definitions \ref{iteratedmodifiedA}, \ref{mixtype})
    \end{itemize}
\end{itemize}

The mixed type is very subtle, since the $A$-cycle integral does not preserve the elliptic property. We need to make some modifications if we want to combine the regularized integral with the $A$-cycle integral.  Such modifications turn out to be related to the contact operators $B$ and $C$. In Subsection \ref{1st2ndgenus1pf}, we define two pure types of partition functions. In Subsection \ref{secmodifiedA}, we introduce the key idea and construction on the modified $A$-cycle integral, and discuss the relation among the various iterated integrals and contact iterations. In Subsection \ref{mixDijk}, we define the mixed partition function, and state the main theorems about contact equation (Theorems \ref{ModifiedAcycleElliptic} and \ref{Contactequation}) which lead to Dijkgraaf's master equation (Theorem \ref{DijkgraafMaster}). In Subsection \ref{proofofpropthm}, we give the proof of Theorems \ref{ModifiedAcycleElliptic} and \ref{Contactequation}.

\subsection{Genus one conformal block}\label{g1block}

In this subsection, we recall the definition of the genus one conformal block \cite{1996Modular}. We first introduce some notations.

For $n\in\Z_{>0}$, let $\mathcal{F}_n$ denote the space of meromorphic functions $f(z_1,\dots, z_n, \tau)$ on $\mathbb{C}^n\times \mathbf{H}$  such that for fixed $\tau$, $f(z_1,\dots, z_n, \tau)$ is doubly periodic for each variable $z_i$ with periods $s+t\tau (s,t\in\mathbb{Z})$ and the possible poles of $f(z_1,\dots, z_n, \tau)$ lie in the set
$$
\{(z_1,\dots, z_n, \tau)|z_i=z_j+(s+t\tau), i\neq j; s,t\in\mathbb{Z}\}.
$$

\begin{defn}\label{genus1cb}
   Let $\{V, |0\rangle, T, Y(, z), \omega\}$ be a conformal vertex algebra. A system of $\operatorname{maps} S=\left\{S_{n}\right\}_{n=1}^{\infty}$
$$S_{n}: V^{\otimes n} \rightarrow {\mathcal{F}}_{n},\quad a_{1}\otimes\cdots\otimes a_{n} \mapsto S\left(\left(a_{1}, z_{1}\right), \ldots,\left(a_{n}, z_{n}\right), \tau\right)
$$
is said to satisfy the \emph{genus-one property} if the following axioms $(1)-(6)$ are satisfied.

(1) $S\left(\left(a_{1}, z_{1}\right),\left(a_{2}, z_{2}\right), \ldots,\left(a_{n}, z_{n}\right), \tau\right)$ is linear for each $a_{i}$;

(2) For every permutation $\sigma$ of $n$ letters,
$$
S\left(\left(a_{\sigma(1)}, z_{\sigma(1)}\right), \ldots,\left(a_{\sigma(n)}, z_{\sigma(n)}\right), \tau\right)=S\left(\left(a_{1}, z_{1}\right), \ldots,\left(a_{n}, z_{n}\right), \tau\right);
$$

(3)
$$
S\left((|0\rangle, z),\left(a_{1}, z_{1}\right), \ldots,\left(a_{n}, z_{n}\right), \tau\right)=S\left(\left(a_{1}, z_{1}\right), \ldots,\left(a_{n}, z_{n}\right), \tau\right);
$$

(4)
$$S\left(\left(L_{-1} a_{1}, z_{1}\right), \ldots,\left(a_{n}, z_{n}\right), \tau\right)=\frac{d}{d z_{1}} S\left(\left(a_{1}, z_{1}\right), \ldots,\left(a_{n}, z_{n}\right), \tau\right);$$

(5)
$$
\begin{array}{l}
\operatorname{Res}_{w=z} S\left((a, w),\left(a_{1}, z_{1}\right), \ldots,\left(a_{n}, z_{n}\right), \tau\right)\left(w-z_{1}\right)^{k} \\
=S\left(\left(a_{(k)}a_{1}, z_{1}\right), \ldots,\left(a_{n}, z_{n}\right), \tau\right);
\end{array}
$$
$\quad$(6) If $a_{i}(i=1, \ldots, n)$ are highest weight vectors for the Virasoro algebra with degree $\operatorname{deg} a_{i},$ write
$$
S=S\left(\left(\omega, w_{1}\right), \ldots,\left(\omega, w_{m}\right),\left(a_{1}, z_{1}\right), \ldots,\left(a_{n}, z_{n}\right), \tau\right),
$$
then
\begin{align*}
&S((\omega,w),(\omega,w_1),\ldots,(\omega,w_m),(a_1,z_1),\ldots,(a_n,z_n),\tau)\\
=~&2\pi i\frac{d}{d\tau}S+\sum_{k=1}^n(\wp_1(w-z_k,\tau)-(w-z_k)\frac{\pi^2}{3}E_2(\tau))\frac{d}{dz_k}S\\
&+\sum_{k=1}^m(\wp_1(w-w_k,\tau)-(w-w_k)\frac{\pi^2}{3}E_2(\tau))\frac{d}{dw_k}S\\
&+\sum_{k=1}^n(\wp_2(w-z_k,\tau)+\frac{\pi^2}{3}E_2(\tau))S\\
&+\sum_{k=1}^m(\wp_2(w-w_k,\tau)+\frac{\pi^2}{3}E_2(\tau))2S\\
&+\frac{1}{2}\sum_{k=1}^m\wp_4(w-w_k,\tau)S((\omega,w_1),\ldots,\widehat{(\omega,w_k)},\ldots,(a_n,z_n),\tau).
\end{align*}
Here $\left(\widehat{\omega, w_{k}}\right)$ means the term is omitted, and
\begin{align*}
&\wp_1(z;\tau)=\frac{1}{z}+\sum_{(m,n)\in(\Z\oplus\Z)\setminus \{(0,0)\}}\left(\frac{1}{z-m\tau-n}+\frac{1}{m\tau+n}+\frac{z}{(m\tau+n)^2}\right),\\
&\wp_{k+1}(z;\tau)=-\frac{1}{k}\frac{d}{dz}\wp_k(z;\tau),\quad\text{for }~k\geq1.
\end{align*}
\end{defn}

The linear space of $S$ satisfying the genus one property is called the \textit{ conformal block on the torus}.  We denote the space of conformal blocks associated to the conformal vertex algebra $(V, |0\rangle, T, Y(, z), \omega)$ by ${\mathcal{H}}_V$.

Here, we give an example of the explicit construction of $S$ as follows.
\subsubsection{Normal ordered Feynman graph}\label{normalorderedFG}
Introduce the free boson $\phi(z)$ with OPE and mode expansion
$$\phi(z)\phi(w)\sim\log(z-w),\quad \phi(z)=\sum_{n\neq0}-\frac{\phi_nz^{-n}}{n}+\phi_0+p,$$
where $p$ is the momentum creation operator as a conjugate of $\phi_0$.

Now we consider a normal ordered Feynman graph in the chiral boson theory on $E_\tau$. Let
$$P(z_1,z_2;\tau)=\wp(z_1-z_2;\tau)+\frac{\pi^2}{3}E_2(\tau)$$
be the propagator. Here $\wp(z;\tau)$ is the Weierstrass $\wp$-function and $E_2$ is the second Eisenstein series.

We also write $P(z_1,z_2;\tau)$ as $P(z_1-z_2;\tau)$ in the following. For convenience, we decompose the propagator $P$ into a singular part and a smooth part
$$P(z_1,z_2)=R(z_1,z_2)+Q(z_1,z_2),\quad \text{for }~ z_1\neq z_2,$$
where
\begin{enumerate}
\setlength{\itemsep}{1pt}
  \item $R(z_1,z_2)=1/(z_1-z_2)^2$;
  \item $Q(z_1,z_2)=\sum_{m\geq0}Q_m(z_1-z_2)^m$ is a power series of $(z_1-z_2)$, more explicitly,
  \begin{align*}
  &Q_0=\frac{\pi^2}{3}E_2(\tau),\qquad Q_{2m+1}=0,\quad m>0,\\
  &Q_{2m}=(2m+1)G_{2m+2}(\tau),\quad m>0,
  \end{align*}
  where $G_{2m+2}(\tau)$ is the Eisenstein series of weight $2m+2$.
\end{enumerate}

Let $\Gamma$ be an oriented graph with possible self-loops (it is not necessarily connected). Let $E(\Gamma)$ be its set of edges, and $V(\Gamma)$ be its set of vertices with $|V(\Gamma)|=n$. We can order the vertices by $\{1,2,\ldots, n\}$.
\begin{itemize}
  \item For the $i$-th vertex with valence $n_i$, we put the normal ordered fields
  $$:\L_i(z_i):=\sum_{k_1,\ldots,k_{n_i}>0}c_{k_1,\ldots,k_{n_i}}:\p^{k_1}\phi(z_i)\p^{k_2}\phi(z_i)\cdots\p^{k_{n_i}}\phi(z_i):,$$
  with only finite $c_{k_1,\ldots,k_{n_i}}$'s are non-vanishing.
  \item For each edge connecting $:\L_i(z_i):$ and $:\L_j(z_j):$, $z_i\neq z_j$, we put $$\sum_{k,l\geq0}\partial_{z_i}^k\partial_{z_j}^lP(z_i,z_j)\frac{\p}{\p(\p^{k+1}\phi(z_i))}\frac{\p}{\p(\p^{l+1}\phi(z_j))}.$$
  \item For each self-loop at the $i$-th vertex, we put
      $$\sum_{k,l\geq0}\partial_{z_i}^k\partial_{z_j}^lQ(z_i,z_j)\bigg|_{z_j=z_i}\frac{\p}{\p(\p^{k+1}\phi(z_i))}\frac{\p}{\p(\p^{l+1}\phi(z_i))}.$$
\end{itemize}

For each graph $\Gamma$, we can define a meromorphic function $\Phi_{\Gamma}(z_1,\ldots,z_n)$ on the configuration space as follows.
\begin{enumerate}
  \item If $\Gamma$ has tails, it is defined as zero.
  \item If $\Gamma$ has no tails, $\Phi_{\Gamma}(z_1,\ldots,z_n)$ is given by the Feynman rule above.
\end{enumerate}

By definition, $\Phi_{\Gamma}(z_1,\ldots,z_n)$ only has poles at the diagonal of the configuration space.
\begin{rmk}
    Similar construction has been discussed in \cite{Tuite2012, Mason2003, Takhtajan2005}. This is a mathematical reformulation of the point-splitting
 method used by physicists. See also \cite{Gui2021, Gui2023} for the case of free $\beta\gamma-bc$ systems.
\end{rmk}

\begin{defn}\label{defnnormalorderedFG} The normal ordered Feynman graph function for arbitrary fields $\L_1(z_1)$,\ldots, $\L_n(z_n)$ is defined to be
$$\langle :\L_1(z_1):,:\L_2(z_2):,\cdots,:\L_n(z_n):;\tau\rangle=\sum_{\Gamma}\frac{\Phi_{\Gamma}(z_1,\ldots,z_n)}{|\Aut(\Gamma)|},$$
where the summation is over all possible graphs as above, and $\Aut(\Gamma)$ is the automorphism group of $\Gamma$.
\end{defn}

\begin{eg} Let us compute $\langle :\p^{i_1}\phi(z_1)\p^{j_1}\phi(z_1):, :\p^{i_2}\phi(z_2)\p^{j_2}\phi(z_2):;\tau\rangle$. The contributing graphs are given by

\begin{figure}[H]
	\centering
	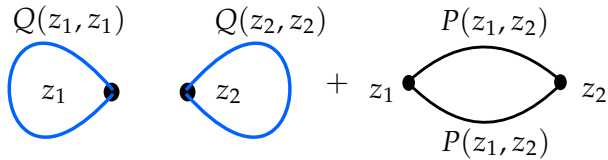
	\caption[]{Example 3.4}
\end{figure}
\vskip -0.5cm
\begin{align*}
 &\langle :\p^{i_1}\phi(z_1)\p^{j_1}\phi(z_1):, :\p^{i_2}\phi(z_2)\p^{j_2}\phi(z_2):;\tau\rangle\\
=~&(-1)^{j_1+j_2-2}Q_{i_1+j_1-2}Q_{i_2+j_2-2}+\big(\p_{z_1}^{i_1-1}\p_{z_2}^{i_2-1}P\big)\big(\p_{z_1}^{j_1-1}\p_{z_2}^{j_2-1}P\big)+\big(\p_{z_1}^{i_1-1}\p_{z_2}^{j_2-1}P)(\p_{z_1}^{j_1-1}\p_{z_2}^{i_2-1}P\big).
\end{align*}
\end{eg}

For the vertex algebra $V$ of chiral bosons
$$V=\Span\left\{\phi_{-i_1}^{k_1}\phi_{-i_2}^{k_2}\cdots\phi_{-i_n}^{k_n}|0\rangle~\big| i_1>i_2\cdots>i_n>0, k_1,\ldots,k_n\geq0, n\geq0\right\},$$
we define
\begin{equation*}
\begin{aligned}
S_n:~V^{\otimes n}~&\longrightarrow~\mathcal{F}_n\\
a_1\otimes\cdots\otimes a_n &\longmapsto S((a_1,z_1),\ldots,(a_n,z_n),\tau)
\end{aligned}
\end{equation*}
by the rescaled normal ordered Feynman graph function
\begin{equation}
S((a_1,z_1),\ldots,(a_n,z_n),\tau)~=~\eta(\tau)^{-1}\langle :\L_1(z_1):,:\L_2(z_2):,\cdots,:\L_n(z_n):;\tau\rangle,
\end{equation}
where $\L_i(z_i)=Y(a_i,z_i)$ and $\eta(\tau)$ is Dedekind eta function
$$\eta(\tau)=q_{24}\prod_{n=1}^{\infty}(1-q^n),\quad q_{24}=e^{\frac{2\pi i\tau}{24}},\quad q=e^{2\pi i\tau}.$$

\begin{prop}\label{normalcorrelation} $S=\{S_n\}_{n\geq1}$ defines an element in the conformal block on the torus.
\end{prop}
\begin{proof} See in Appendix \ref{proofnormal}.
\end{proof}

\subsection{Regularized integrals}\label{regularizedintegrals}
In this subsection, we will recall the necessary definition and results of regularized integrals introduced in \cite{SiJie2020} and apply them to the conformal block to obtain a mathematical definition of the genus one partition functions in \cite{Dijk1996master}.

Let $\Sigma$ be a smooth Riemann surface, and denote
\begin{itemize}
  \item $\Delta:=\cup_{i\neq j}\Delta_{ij}$, $\Delta_{ij}=\left\{(z_1,z_2,\ldots,z_n)\in\Sigma^n|z_i=z_j\right\}$.
  \item $A^{p,q}_{\Sigma^n}(\star\Delta)$: sheaf of smooth $(p,q)$-forms with arbitrary holomorphic poles along $\Delta$.
  \item $A^{p,q}_{\Sigma^n}(\log\Delta)$: sheaf of smooth $(p,q)$-forms with possible holomorphic log-poles along $\Delta$.
\end{itemize}

\begin{rmk}
    We follow the notation of \cite{SiJie2020}. A local section of $A^{p,q}_{\Sigma^n}(\log\Delta)$ is a linear combination of forms of type $\omega=\nu\cdot \frac{df_1}{f_1}\wedge\cdots \wedge \frac{df_l}{f_l}$ where $\nu$ is a smooth form and these $f_i$'s are local equations of some irreducible components of the diagonal divisor ($d$ is the holomorphic differential). For the detailed study of the logarithmic de Rham complex, see \cite[4.2.13]{Beilinson2004}.
\end{rmk}
For any $\omega\in A^{n,\bullet}_{\Sigma^n}(\star\Delta)$, there exists a decomposition (by reduction of poles)
$$\omega=\alpha+\partial\beta$$
where $\alpha\in A^{n,\bullet}_{\Sigma^n}(\log \Delta)$, $\beta\in A^{n-1,\bullet}_{\Sigma^n}(\star\Delta)$.
Then they define the regularized integral
\begin{equation}\label{regint}
\dashint_{\Sigma^n}\omega:=\int_{\Sigma^n}\alpha.
\end{equation}

The notion of regularized integrals satisfies nice properties, for example, see \cite{SiJie2020},
\begin{itemize}
  \item it works for Riemann surfaces with boundary;
  \item Stokes' theorem, Fubini theorem, Riemann-Hodge bilinear relation;
  \item There is a relative version.
\end{itemize}

Here we give the explicit description of Stokes' formula and Fubini Theorem for the regularized integrals, which we will mainly use in the following.

\begin{thm}[\cite{SiJie2020}, Theorem 2.13]\label{Stokes}
 Let $\Sigma$ be a compact Riemann surface possibly with boundary $\partial \Sigma$. Let $\alpha \in A_{\Sigma}^{1}(\star D)$, where $D\subset\Sigma$ is a subset of finite points which does not meet $\partial\Sigma$. Then we have the following version of Stokes' formula for the regularized integral
$$
\dashint_{\Sigma} d \alpha=-\sum_{p\in D}\operatorname{Res}_{z=p}(\alpha)+\int_{\partial \Sigma} \alpha.
$$
\end{thm}

\begin{thm}[Fubini Theorem: \cite{SiJie2020} Theorem 2.37, Corollary 2.39]\label{Fubini} Let $\sigma$ be any permutation of $\{1,2,\ldots,n\}$. Then the iterated regularized
integral
$$\dashint_{\Sigma_{\sigma(1)}}\dashint_{\Sigma_{\sigma(2)}}\cdots\dashint_{\Sigma_{\sigma(n)}}\omega,\quad\omega\in A_{\Sigma^n}^{2n}(\star\Delta)$$
does not depend on the choice of $\sigma$.
\end{thm}
\begin{rmk}
    The regularized integral can be viewed as an explicit realization of the trace map on the unit chiral chain complex \cite[4.3.3]{Beilinson2004}. This connection is established in \cite{Gui2021}.
\end{rmk}

In the rest of this section, we apply the theory of regularized integrals to the elliptic curve $E_{\tau}$
$$
E_{\tau}=\mathbb{C}/(\mathbb{Z}+\mathbb{Z}\tau),\quad \tau\in\mathbf{H}.
$$
Recall that $z$ is a linear holomorphic coordinate on the universal cover $\mathbb{C}$ in terms of which one has
$$
\int_{E_\tau}\frac{d^2z}{\operatorname{Im}\tau}=1,\quad d^2z:=\frac{\sqrt{-1}}{2}dz\wedge d\bar{z}.
$$
\begin{defn}\label{Almost-Meromorphic}
We say a function $\Psi$ on $\mathbb{C}^{n} \times \mathbf{H}$ is \emph{almost-meromorphic} if $\Psi$ can be written as a finite
sum
$$
\Psi\left(z_{1}, \cdots, z_{n} ; \tau\right)=\sum_{k_{1}, \cdots, k_{n}, m\geq0} \Psi_{k_{1}, \cdots, k_{n} ; m}\left(z_{1}, \cdots, z_{n} ; \tau\right)\left(\frac{\operatorname{Im} z_{1}}{\operatorname{Im} \tau}\right)^{k_{1}} \cdots\left(\frac{\operatorname{Im} z_{n}}{\operatorname{Im} \tau}\right)^{k_{n}}\left(\frac{1}{\operatorname{Im} \tau}\right)^{m}
$$
where each $\Psi_{k_{1}, \cdots, k_{n} ; m}\left(z_{1}, \cdots, z ; \tau\right)$ is a meromorphic function on $\mathbb{C}^{n} \times \mathbf{H}$.
\end{defn}

\begin{defn}\label{Elliptic-HolomorphicAway}
 Let $\mathcal{R}_{n}^{E_{\tau}}$ denote the space of functions $\Psi$ on $\mathbb{C}^{n} \times \mathbf{H}$ such that
\begin{itemize}
  \item $\Psi$ \text { is elliptic and almost-meromorphic. }
  \item { Each component } $\Psi_{k_{1}, \cdots, k_{n} ; m}$ \text { as in Definition \ref{Almost-Meromorphic}} \text { is holomorphic away from diagonals. }
\end{itemize}
\end{defn}

\begin{prop}[\cite{SiJie2020}, Proposition 3.15]
 The regularized integral defines a map
$$
\dashint_{E_{\tau}} \frac{d^{2} z_{n}}{\operatorname{Im} \tau}: \mathcal{R}_{n}^{E_{\tau}} \rightarrow \mathcal{R}_{n-1}^{E_{\tau}}.
$$
In particular, for $\Psi\in\mathcal{R}_n^{E_{\tau}}$,
\begin{equation}\label{keyeq}
\dashint_{E_{\tau}}\frac{d^{2} z_{n}}{\im\tau}\Psi(z)=\oint_Adz_n\Psi(z)+\sum_{j\neq n}\operatorname{Res}_{z_n=z_j}\frac{\im z_n}{\im\tau}\Psi(z).
\end{equation}
\end{prop}

\subsection{The two pure types of genus one partition functions}\label{1st2ndgenus1pf}

Now we will define two pure types of genus one correlators associated with $S\in\mathcal{H}_V$.

\begin{defn} Let $m\in\mathbb{N}$ and $S\in\mathcal{H}_V$. We define the first type of $m$-point correlators $\langle-\rangle_{E_{\tau},m}$ as follows:
\begin{equation}\label{1stkind}
\begin{aligned}
\langle-\rangle_{E_{\tau},m}:&~\mathcal{A}^{\otimes m}\longrightarrow\C\\
\langle a_1,\ldots,a_m\rangle_{E_{\tau},m}
=~&\prod_{k=1}^m\dashint_{E_{\tau}}\frac{d^2z_k}{\im\tau}S((a_1,z_1),\ldots,(a_m,z_m),\tau).
\end{aligned}
\end{equation}
By choosing a countable number of elements $\{v_i\}_{i\in\N}$ in Definition \ref{intcdt}, the first type of genus one partition function $Z_{E_{\tau}}(t)$ is defined to be
$$Z_{E_{\tau}}(t)=\sum_{m\geq0}\frac{1}{m!}\sum_{(i_1,\ldots,i_m)\in\N^m}~t^{i_1}\cdots t^{i_m}\langle v_{i_1},\ldots,v_{i_m}\rangle_{E_{\tau},m},$$
where $\{t^{i}\}$ are the coordinates of $\{v_i\}$ respectively.
\end{defn}

The correlators are well-defined by Fubini Theorem \ref{Fubini} of the regularized integral.

\vskip 0.2cm

On the other hand,  the iterated $A$-cycle integral in general depends on the order of integration. However, Dijkgraaf's integrable condition guarantees the following result.

\begin{lem}\label{Commutative} Let $\mathcal{A}$ be the subvertex algebra satisfying the Dijkgraaf's condition and let $S\in\mathcal{H}_V$. For any $a\in\mathcal{A}^{\otimes n}$, we have
$$\oint_Adz_i\oint_Adz_jS(a)=\oint_Adz_j\oint_Adz_iS(a).$$
\end{lem}
\begin{proof}
We use the axiom (5) for $S$
\begin{align*}
  \oint_Adz_i\oint_Adz_jS(a)-\oint_Adz_j\oint_Adz_iS(a) &=\oint_{A}dz_{j}\operatorname{Res}_{z_i=z_j}S(a) \\
   & =\oint_{A}dz_{j}\partial_{z_j}S(B_{i\rightarrow j}a)\\
   &=0.\qedhere
\end{align*}
\end{proof}

Thus, we arrive at the following definition.

\begin{defn}Let $n\in\mathbb{N}$ and $S\in\mathcal{H}_V$. We define the second type of $n$-point correlators $\langle-\rangle_{A,n}$ as follows:
\begin{equation}\label{2ndkind}
\begin{aligned}
\langle-\rangle_{A,n}:&~\mathcal{A}^{\otimes n}\longrightarrow\C\\
\langle a_1,\ldots,a_n\rangle_{A,n}
=~&\prod_{l=1}^n\oint_A{dz_l}S((a_1,z_1),\ldots,(a_n,z_n),\tau).
\end{aligned}
\end{equation}
Similarly, choosing $\{v_i\}$ as above, the second kind of genus one partition function $Z_{A}(s)$ is defined to be
$$Z_{A}(s)=\sum_{n\geq0}\frac{1}{n!}\sum_{(j_1,\ldots,j_n)\in\N^n}s^{j_1}\cdots s^{j_n}\langle v_{j_1},\ldots,v_{j_n}\rangle_{A,n},$$
where $\{s^j\}$ are the coordinates of $\{v_j\}$ respectively.
\end{defn}

We are interested in the relation between the first type and the second type correlators. Inspired by Dijkgraaf's work, we need to define a more general partition function $Z(s,t)$ including arbitrary iterated $A$-cycle integrals and iterated regular integrals. However, as we explained above, $\oint_A$ does not preserve the ellipticity in general, thus the mixed double integral $\dashint_{E_{\tau}}\oint_A$ does not make sense. In order to obtain a well-defined genus one partition function $Z(s,t)$ of mixed type, we need to do some modification of $A$-cycle integral with the help of the following equation (i.e. \eqref{keyeq} with $\Psi(z)$ being $S(a)$):
\begin{equation}\label{keyeq2}
\dashint_{E_{\tau}}\frac{d^{2} z_{i}}{\im\tau}S(a)=\oint_Adz_i S(a)+\sum_{j\neq i}\Res_{z_i=z_j}\frac{\im z_i}{\im\tau}S(a).
\end{equation}

By the relation of operations $B$, $C$ and the axiom of $S$, we have the nice decomposition for the residue term in \eqref{keyeq2}.
\begin{lem}\label{OPE} Let $S\in\mathcal{H}_V$, for any $a\in\mathcal{A}^{\otimes n}$, we have
\begin{equation}\label{modfiedBCD}
\Res_{z_i=z_j}\frac{\im z_i}{\im\tau}S(a)=\partial_{z_j}\left(\frac{\im z_j}{\im\tau}S(B_{i\rightarrow j}a)\right)+\frac{1}{2\sqrt{-1}\im\tau}S(C_{i\rightarrow j}a).
\end{equation}
\end{lem}
\begin{proof} This equation could also be decomposed into two equations
\begin{align*}
  \p_{z_j}S(B_{i\rightarrow j}a)&=\operatorname{Res}_{z_i=z_j}S(a),\\
  \operatorname{Res}_{z_i=z_j}(z_i-z_j)S(a)&=S(B_{i\rightarrow j}a)+S(C_{i\rightarrow j}a),
\end{align*}
which follow from the axioms (4) (5) for $S$ and the definition of $B,C$.
\end{proof}

The relation \eqref{BC2} of $B$ and $C$ also has an expression in terms of the $A$-cycle integral.
\begin{lem}\label{BCS}  Let $S\in\mathcal{H}_V$, for any $a\in {\mathcal{A}}^{\otimes n}$, we have
$$\oint_Adz_jS(C_{i\rightarrow j}a)=\oint_Adz_i S(B_{j\rightarrow i}a).$$
\end{lem}
\begin{proof} By \eqref{BC2}, we have
$$C(a_i,a_j)=B(a_j,a_i)+Tb,\quad \text{for some }b\in\mathcal{A}.$$
Since $\oint_Adz_j \partial_{z_j}S(a_1\otimes\underset{\text{i-th}}{|0\rangle}\otimes\underset{\text{j-th}} {b}\otimes\cdots\otimes a_n)=0$, and using the axioms (2) (3), we get
\begin{align*}
   \oint_Adz_j S(C_{i\rightarrow j}a)&=\oint_Adz_j S(a_{1}\otimes \displaystyle \underset{\text{i-th}}{|0\rangle}\otimes\cdots \otimes\underset{\text{j-th}} {B(a_j,a_i)} \otimes \cdots\otimes a_{n})
 \\
   &= \oint_Adz_i S(a_{1}\otimes \displaystyle \underset{\text{i-th}}{B(a_j,a_i)}\otimes\cdots \otimes\underset{\text{j-th}} {|0\rangle} \otimes \cdots\otimes a_{n})\\
   &= \oint_Adz_i S(B_{j\rightarrow i}a).\qedhere
\end{align*}
\end{proof}

\begin{thm} For any $a\in\mathcal{A}^{\otimes l}$, and 
any $0<m,n<l$,  $$
\dashint_{E_{\tau}^{m}}S(a)\frac{dz_{j_1}^2}{\im\tau}\cdots\frac{dz_{j_m}^2}{\im\tau}\quad  \text{and} \quad \oint_{A^{n}}S(a)dz_{k_1}\cdots dz_{k_n}$$ 
satisfy the axioms (1)-(5) of the genus one conformal block.
\end{thm}

\begin{proof} First, we look at the case $m,n=1$. It is obvious that both $\dashint_{E_{\tau}}\frac{dz_j^2}{\im\tau}S(a)$ and $\oint_Adz_kS(a)$ satisfy (1)-(3).
\begin{itemize}
    \item For $\dashint_{E_{\tau}}\frac{dz_j^2}{\im\tau}S(a)$,  
    \begin{itemize}
        \item the axiom (4) follows from \cite{SiJie2020} Lemma 2.26 that
        $$\left[\frac{d}{dz_i},~\dashint_{E_{\tau}}\frac{dz_j^2}{\im\tau}\right]=0,\quad i\neq j.$$ 
        \item the axiom (5) follows from \cite{SiJie2023} Corollary A.2 that
        \[\left[\Res_{z_a=z_b},~\dashint_{E_{\tau}}\frac{dz_j^2}{\im\tau}\right]=0,\quad\text{for distinct } ~a,b,j.\]
    \end{itemize}
    \item For $\oint_Adz_kS(a)$,
    \begin{itemize}
        \item the axiom (4) follows from the obvious identity 
        \[\left[\frac{d}{dz_i},~\oint_Adz_k\right]=0,\quad i\neq k.\]
        \item the axiom (5) follows from \cite{SiJie2023} Lemma A.1 that
        \[\left[\Res_{z_a=z_b},~\oint_Adz_k\right]=0,\quad \text{for distinct }~a,b,k.\]
    \end{itemize}
\end{itemize}
Then the desired results hold by induction.
\end{proof}
\begin{rmk}
The regularized integral $\dashint_{E_{\tau}^{m}}S(a)\frac{dz_{j_1}^2}{\im\tau}\cdots\frac{dz_{j_m}^2}{\im\tau}$ is not a genus one conformal block in the strict sense, as it involves $\im\tau$. However, since it is elliptic, one may interpret it as a smooth section of the bundle of (co)invariants over the moduli space of elliptic curves with marked points. The $A$-cycle integral $\oint_{A^{n}}S(a)dz_{k_1}\cdots dz_{k_n}$ is meromorphic but not elliptic, one can add corrections to make it elliptic, see Theorem \ref{ModifiedAcycleElliptic}.
\end{rmk}

\subsection{The modified $A$-cycle integrals}\label{secmodifiedA}
\subsubsection{Modified $A$-cycle integrals}
Now, by Equations \eqref{keyeq2} and \eqref{modfiedBCD}, we have
\begin{equation}\label{EmodifA}
\dashint_{E_{\tau}}\frac{d^{2} z_{i}}{\im\tau}S(a)=\oint_Adz_i S(a)+\sum_{j\neq i}\partial_{z_j}\left(\frac{\im z_j}{\im\tau}S(B_{i\rightarrow j}a)\right)+\sum_{j\neq i}\frac{1}{2\sqrt{-1}\im\tau}S(C_{i\rightarrow j}a).
\end{equation}

Thus, we can define the modified $A$-cycle integral by
\begin{equation}\label{modifiedAcycle}
\widehat{\oint}_Adz_iS(a)=\oint_Adz_i S(a)+\sum_{j\neq i}\partial_{z_j}\left(\frac{\im z_j}{\im\tau}S(B_{i\rightarrow j}a)\right).
\end{equation}
By Equation \eqref{EmodifA}, $\widehat{\oint}_A$ preserves the ellipticity. Unfortunately,
$$\left[\widehat{\oint}_{A}dz_n, \widehat{\oint}_{A}dz_m\right]S(a)\neq0.$$
There is a contact interaction between the two modified $A$-cycle integrals. So we need to make the second modification.

\begin{defn}\label{defn:notation}
For convenience, we simplify some notations and denote
\begin{itemize}
\setlength{\itemsep}{2pt}
\setlength{\parsep}{0pt}
\setlength{\parskip}{0pt}
\item For $I\subset\{1,\ldots,n\}$,
 $$E_I:=\dashint_{E_{\tau}^{|I|}}\prod_{i\in I}\frac{d^2z_i}{\operatorname{Im}\tau},\qquad A_I:=\prod_{i\in i}\oint_Adz_i$$
\item $\widehat{A}_i=\widehat{\oint}_Adz_i$.
\end{itemize}
\end{defn}

Now we want to define a modified (multi-)$A$-cycle integral $\widehat{A}_I$ associated to any subset $I\subset\{1,\ldots,n\}$, which preserves the ellipticity of functions. For $|I|=1$, this is done. Next we consider $I=\{1,2\}$ containing two indices.

\begin{defn} For $S\in\mathcal{H}_V$ and $a\in \mathcal{A}^{\otimes n}$, we define
\begin{equation}\label{modifiedA2}
\begin{aligned}
\widehat{A}_{\{1,2\}}S(a):=~&\widehat{\oint}_{A_1\times A_2}dz_1dz_2S(a)\\
                      =~&\oint_{A_1}dz_1\oint_{A_2}dz_2S(a)+\sum_{j\neq 1,2}\partial_{z_j}\left[\frac{\im z_j}{\im\tau}\oint_{A_2}S(B_{1\rightarrow j}a)\right]\\
                      &+\sum_{k\neq 1,2}\partial_{z_k}\left[\frac{\im z_k}{\im\tau}\oint_{A_1}S(B_{2\rightarrow k}a)\right]\\
                      &+\sum_{j,k\neq 1,2;j\neq k}\partial_{z_j}\left[\frac{\im z_j}{\im\tau}\partial_{z_k}\left(\frac{\im z_k}{\im\tau}S(B_{1\rightarrow j}B_{2\rightarrow k}a)\right)\right]\\
                      &+\sum_{k\neq 1,2}\partial_{z_k}\left[\left(\frac{\im z_k}{\im\tau}\right)^2\partial_{z_k}S(B_{1\rightarrow k}B_{2\rightarrow k}a)\right].
\end{aligned}
\end{equation}
\end{defn}

Inductively, one can define $\widehat{A}_I$ for any $I\subset\{1,2,\ldots,n\}$. To give the explicit definition, let us look at the terms in Equation \eqref{modifiedA2} carefully and introduce the following differential operators $D_{r,k}$, $r\geq1$,
$$D_{r,k}F:=\partial_{z_k}\left(\left(\frac{\im z_k}{\im\tau}\right)^r\partial_{z_k}^{r-1}F\right),\quad \text{for any }F\in\mathcal{R}_n^{E_{\tau}}.$$
Similarly, we will consider a general version of $D_{r,k}$ associated to any two disjoint sets $I$, $M$ and a map $\phi:I\rightarrow M$
$$D_{I\rightarrow M}^{\phi}:=\prod_{k\in M,\phi^{-1}(k)\neq \emptyset}D_{|\phi^{-1}(k)|,k}.$$

Now let $M$, $N$, $I$ and $I^c$ be sets as follows:
$$M=\{1,\ldots,m\},\quad N=\{m+1,\ldots,m+n\},\quad I\subset N,\quad I^c=N-I.$$
For $S\in \mathcal{H}_V$, we also use the notation $[a]$ instead of $S(a)$ for $a\in \mathcal{A}^{\otimes (m+n)}.$

\begin{defn}\label{iteratedmodifiedA} The modified $A$-cycle integral $\widehat{A}_N$ for the holomorphic conformal block on the torus is defined to be
$$
\widehat{A}_{N}[a]=\displaystyle \sum _{ \substack{I\subset N\\
\phi:I\rightarrow M}}D^{\phi}_{I\rightarrow M}A_{{I}^c} [B^{\phi}_{{I\rightarrow M}}a], \quad a\in \mathcal{A}^{\otimes (|M|+|N|)},\  I^c=N-I.
$$
In particular,
\begin{enumerate}
  \item $\lim\limits_{\im\tau\rightarrow\infty}\widehat{A}_N[a]=A_N[a]$.
  \item when $M=\emptyset$,
  $$\widehat{A}_N[a]=A_{N}[a]=\oint_Adz_1\cdots \oint_Adz_n[a].$$
\end{enumerate}

\end{defn}
\begin{rmk}
    It is clear that $\widehat{A}_N[a]$ can be written as $$\widehat{A}_N[a]=\sum\limits^{N_0}_{r=0}\frac{F_r(z,\tau)}{(\im\tau)^r},\quad\text{for some } N_0>0,$$
    where each $F_r(z,\tau)$ is meromorphic in $\tau$. We define the limit $\lim\limits_{\im\tau\rightarrow\infty}\widehat{A}_N[a]$ to be the constant term $F_0(z,\tau)$.
\end{rmk}

\subsubsection{Relations among the various iterated integrals $E_M$, $\widehat{A}_N$, $A_N$ and contact iterations}

Next we consider a generalization of \eqref{EmodifA}:
$$E_i[a]=\widehat{A}_i[a]+\sum_{j\neq i}\frac{1}{2\sqrt{-1}\im\tau}[C_{i\rightarrow j}a]$$
as in the next Theorem \ref{ModifiedAcycleElliptic}, which  explains the advantage of modified $A$-cycle integrals.

\begin{thm}\label{ModifiedAcycleElliptic} For any $a\in \mathcal{A}^{\otimes(|M|+|N|)}$, $\widehat{A}_{N}[a]$ is almost-meromorphic and elliptic, also satisfies the axioms (1)-(5) of the genus one conformal block. And we have
  $$
  E_{m}\widehat{A}_{N}[a]=\widehat{A}_{N\cup\{m\}}[a]+\frac{1}{2\sqrt{-1}\operatorname{Im}\tau}\sum_{1\leq j\leq m+n,j\neq m}\widehat{A}_{N}[C_{m\rightarrow j}a].
  $$
\end{thm}

\begin{proof} See in Subsection \ref{proofofpropthm}.
\end{proof}

\begin{cor}
  For $a\in \mathcal{A}^{\otimes (m+l)}$, we have 
  $$
  E_{\{1,\dots, m\}}[a]=\sum_{k\geq 0}\frac{1}{(2\sqrt{-1}\operatorname{Im}\tau)^k}\displaystyle \sum _{ \substack{I=\{i_1,\dots,i_k\}\\1\leq i_1<\cdots<i_k\leq m\\ 1\leq j_1,\dots,j_k\leq m+l}}\widehat{A}_{I^c}[C_{i_1\rightarrow j_1}\cdots C_{i_k\rightarrow j_k}a],
  $$
  here $C_{i\rightarrow i}:=0$, for any $i\leq m , I^c=\{1,\dots,m\}-I$. In particular, when $l=0, a\in\mathcal{A}^{\otimes m}$, we have
  $$
  E_{\{1,\dots, m\}}[a]=\sum_{k\geq 0}\frac{1}{(2\sqrt{-1}\operatorname{Im}\tau)^k}\displaystyle \sum _{ \substack{I=\{i_1,\dots,i_k\}\\1\leq i_1<\cdots<i_k\leq m\\ 1\leq j_1,\dots,j_k\leq m}}{A}_{I^c}[C_{i_1\rightarrow j_1}\cdots C_{i_k\rightarrow j_k}a].
  $$
\end{cor}
\begin{proof}
  We prove this corollary by induction on $m$. If $m=1$, we apply Proposition \ref{ModifiedAcycleElliptic} to the case when $N=\emptyset.$ Now suppose $m>1$, by induction hypothesis and Proposition \ref{ModifiedAcycleElliptic}, we have
  \begin{align*}
       E_1E_{\{2,\dots, m\}}[a]&=E_1\sum_{k\geq 0}\frac{1}{(2\sqrt{-1}\operatorname{Im}\tau)^k}\displaystyle \sum _{ \substack{I=\{i_1,\dots,i_k\}\\2\leq i_1<\cdots<i_k\leq m\\ 1\leq j_1,\dots,j_k\leq m+l}}{\widehat{A}}_{\{2,\dots,m\}-I}[C_{i_1\rightarrow j_1}\cdots C_{i_k\rightarrow j_k}a]\\
     &= \sum_{k\geq 0}\frac{1}{(2\sqrt{-1}\operatorname{Im}\tau)^k}\displaystyle \sum _{ \substack{I=\{i_1,\dots,i_k\}\\2\leq i_1<\cdots<i_k\leq m\\ 1\leq j_1,\dots,j_k\leq m+l}}{\widehat{A}}_{\{1,2,\dots,m\}-I}[C_{i_1\rightarrow j_1}\cdots C_{i_k\rightarrow j_k}a]\\
     &+\sum_{k\geq 0}\frac{1}{(2\sqrt{-1}\operatorname{Im}\tau)^{k+1}}\displaystyle \sum _{ \substack{I=\{i_1,\dots,i_k\}\\2\leq i_1<\cdots<i_k\leq m\\ 1\leq j, j_1,\dots,j_k\leq m+l}}{\widehat{A}}_{\{2,\dots,m\}-I}[C_{1\rightarrow j}C_{i_1\rightarrow j_1}\cdots C_{i_k\rightarrow j_k}a]\\
     &=\sum_{k\geq 0}\frac{1}{(2\sqrt{-1}\operatorname{Im}\tau)^k}\displaystyle \sum _{ \substack{I=\{i_1,\dots,i_k\}\\1\leq i_1<\cdots<i_k\leq m\\ 1\leq j_1,\dots,j_k\leq m}}{\widehat{A}}_{\{1,\dots,m\}-I}[C_{i_1\rightarrow j_1}\cdots C_{i_k\rightarrow j_k}a].\qedhere
  \end{align*}
\end{proof}

\subsection{The mixed type partition function and Dijkgraaf's master equation}\label{mixDijk}
\begin{defn}\label{mixtype}
 We define the genus one correlators of mixed type associated to an element $[-]\in \mathcal{H}_V$  as follows
\begin{align*}
 \langle -\rangle_{m,n}:&~{\mathcal{A}}^{\otimes(m+n)}\longrightarrow \mathbb{C},\\
    \langle a\rangle_{m,n}:&~=~  E_M\widehat{A}_{N}[a],
\end{align*}
where $a\in {\mathcal{A}}^{\otimes (m+n)}$, $M=\{1,\dots,m\}$, $N=\{m+1,\dots, m+n\}.$
\end{defn}

\begin{thm}\label{Contactequation}
Let  $v=v_{k_1}\otimes \cdots\otimes v_{k_{m+n}}$ with $v_{k_i}$ the element in the Dijkgraaf's condition. Then we have
$$
\langle v\rangle_{m,n}=\langle v\rangle_{m-1,n+1}+\frac{1}{2\sqrt{-1}\operatorname{Im}\tau}\sum_{1\leq j\leq m+n,j\neq m}\sum_{l\in\mathbb{N}}c^l_{k_mk_j}\langle v_{k_1}\otimes \displaystyle \underset{\text{m-th}}{1}\otimes\cdots \otimes\underset{\text{j-th}} {v_l} \otimes \cdots\otimes v_{k_{m+n}}\rangle_{m,n}.
$$
\end{thm}
We will prove this theorem in the next subsection.

\begin{defn}

We define the generating function of $\langle v\rangle_{m,n}$, i.e., the genus one partition function of mixed type to be
\begin{equation}\label{Zst}
Z(s,t):=\sum_{m\geq0}\sum_{n\geq0}\frac{1}{m!}\frac{1}{n!}\left\langle \left(\sum_{\alpha}{t^{\alpha}v_{\alpha}}\right)^{\otimes m}\otimes\left(\sum_{\beta}{s^{\beta}v_{\beta}}\right)^{\otimes n}\right\rangle_{m,n},
\end{equation}
where the two summations $\sum$ are taken over all the countable elements ${v_i}$ in the Dijkgraaf's condition, and $t^{\alpha}$, $s^{\beta}$ are formal variables.
\end{defn}

Theorem \ref{Contactequation} implies the following Dijkgraaf's master equation \cite{Dijk1996master}.
\begin{thm}\label{DijkgraafMaster}
The partition function $Z=Z(s,t)$ satisfies the following linear differential equation
\begin{equation}\label{Dijkmast}
\frac{\partial Z}{\partial t^i}=\left[\frac{\partial}{\partial s^i}+\frac{1}{2\sqrt{-1}\operatorname{Im}\tau}(L^{(s)}_i+L^{(t)}_i)\right]Z, \quad \forall ~i\in \mathbb{N},
\end{equation}
here
$$
L^{(s)}_i=\sum_{j,k\in\mathbb{N}}c^k_{ij}s^j\frac{\partial}{\partial s^k}, \quad L^{(t)}_i=\sum_{j,k\in\mathbb{N}}c^k_{ij}t^j\frac{\partial}{\partial t^k}.
$$
\end{thm}

\begin{proof} For any $m,n\geq0$, applying $\prod_{i=1}^m\prod_{j=1}^{n}\frac{\p}{\p t^{\alpha_i}}\frac{\p}{\p s^{\beta_j}}\bigg|_{s=t=0}$ to both sides of \eqref{Dijkmast}, we arrive at the identity of the form in Theorem \ref{Contactequation}.
\end{proof}

\subsection{Proof of Theorem \ref{ModifiedAcycleElliptic} and Theorem \ref{Contactequation}}\label{proofofpropthm}
In this subsection, we complete the proof of Theorem \ref{ModifiedAcycleElliptic} and Theorem \ref{Contactequation}. Throughout this section we denote
\begin{align*}
&M=\{1,\ldots,n\},\quad N=\{m+1,\ldots,m+n\},\\
&M'=M-\{m\},\quad M_k'=M'-\{k\},~\text{ for any }k\in M'\\
&I^c:=N-I,~\text{ for any } I\subset N.
\end{align*}

Since the computation of the proof is cumbersome, we explain the proof idea first. We want to prove Theorem \ref{ModifiedAcycleElliptic} by induction on $|N|$.
\begin{itemize}
  \item The first step ($|N|=0$) for induction is trivial.
  \item For the second step, we assume that $\hat{A}_N[a]$ is almost-meromorphic and elliptic, satisfies the axioms (1)-(5), and then prove the following identity holds
$$E_{m}\widehat{A}_{N}[a]=\widehat{A}_{N\cup\{m\}}[a]+\frac{1}{2\sqrt{-1}\operatorname{Im}\tau}\sum_{0\leq j\leq m+n,j\neq m}\widehat{A}_{N}[C_{m\rightarrow j}a].$$
It follows that $\widehat{A}_{N\cup\{m\}}[a]$ is almost-meromorphic and elliptic, satisfies the axioms (1)-(5).
\end{itemize}

One of the key points in the second step is that we have to compute the surface integral for an almost-meromorphic and elliptic function. There is such a formula in Lemma 3.26 \cite{SiJie2020}. But according to our definition of the modified $A$-cycle integral $\hat{A}_N$, we need a modified form of that formula. For completeness, we write our formula as follows and prove it based on the methods in Lemma 3.26 \cite{SiJie2020}.

\begin{lem}\label{RegIntegralF}
  Let $F(z_1,\dots,z_m;\tau)$ be an almost-meromorphic elliptic function on $\mathbb{C}^{m}\times \mathbf{H}$. Assume that $F$ can be written as
   $$
  F=F_0+\sum^n_{l=1}\partial_{z_m}\left(\left(\frac{\operatorname{Im}z_m}{\operatorname{Im}\tau}\right)^{l}\partial^{l-1}_{z_m}F_l\right),
  $$
  where $F_0(-,z_m;\tau),\dots, F_n(-,z_m;\tau)$ are meromorphic functions on $\mathbb{C}\times \mathbf{H}$ . Moreover, we assume that each $F_i$ is holomorphic way from diagonals. Then
$$
E_mF=\oint_Adz_m\left(F_0+\frac{1}{2\sqrt{-1}\operatorname{Im}\tau}F_1\right)+\sum_{1\leq k\leq m-1}\operatorname{Res}_{z_m=z_k}\left(\frac{\operatorname{Im}z_m}{\operatorname{Im}\tau}F_0\right).
$$

  \end{lem}
\begin{proof}
  We rewrite $F$ as
  $$
  F=F_0+\frac{1}{2\sqrt{-1}\operatorname{Im}\tau}F_1+\sum^{n-1}_{l=1} \left(\frac{\operatorname{Im}z_m}{\operatorname{Im}\tau}\right)^l\left(\partial^l_{z_m}F_{l}+\frac{l+1}{2\sqrt{-1}\operatorname{Im}\tau}\partial^l_{z_m}F_{l+1}\right)+\left(\frac{\operatorname{Im}z_m}{\operatorname{Im}\tau}\right)^n\partial^n_{z_m}F_n.
  $$
  Let
  \begin{align*}
  \Psi=~&\frac{\operatorname{Im}z_m}{\operatorname{Im}\tau}\left(F_0+\frac{1}{2\sqrt{-1}\operatorname{Im}\tau}F_1\right)+\sum^{n-1}_{l=1} \frac{1}{l+1}\left(\frac{\operatorname{Im}z_m}{\operatorname{Im}\tau}\right)^{l+1}\left(\partial^l_{z_m}F_{l}+\frac{l+1}{2\sqrt{-1}\operatorname{Im}\tau}\partial^l_{z_m}F_{l+1}\right)\\
  &+\frac{1}{n+1}\left(\frac{\operatorname{Im}z_m}{\operatorname{Im}\tau}\right)^{n+1}\partial^n_{z_m}F_n,
  \end{align*}
  then
  $$
  \Psi=\frac{\operatorname{Im}z_m}{\operatorname{Im}\tau}F_0+\sum^n_{l=1}\partial_{z_m}\left(\frac{1}{l+1}\left(\frac{\operatorname{Im}z_m}{\operatorname{Im}\tau}\right)^{l+1}\partial^{l-1}_{z_m}F_l\right), \quad \partial_{\bar{z}_m}\Psi=\frac{\sqrt{-1}}{2\operatorname{Im}\tau}F.
  $$
  It follows that 
\begin{itemize}
    \item the relation $F(z_m+1)=F(z_m)$ implies $\Psi(z_m+1)=\Psi(z_m)$;
    \item the relation $F(z_m+\tau)=F(z_m)$ implies
  $$
 \partial_{\bar{z}_m}( \Psi(z_m+\tau)-\Psi(z_m))=\frac{\sqrt{-1}}{2\operatorname{Im}\tau}(F(z+\tau)-F(z))=0.
  $$
  This further implies
  $$
  \Psi(z_m+\tau)-\Psi(z_m)=F_0+\frac{1}{2\sqrt{-1}\operatorname{Im}\tau}F_1+\sum^{n-1}_{l=1} \frac{\partial_{z_m}^lF_{l}+\frac{l+1}{2\sqrt{-1}\operatorname{Im}\tau}\partial_{z_m}^lF_{l+1}}{l+1}+\frac{\partial^n_{z_m}F_n}{n+1}.
  $$
\end{itemize}
  
Now using Theorem \ref{Stokes} (Stokes' Theorem), we have
\begin{align*}
   E_mF=& \dashint_{E_{\tau}}\frac{d^2z_m}{\operatorname{Im}\tau}F=-\dashint_{E_{\tau}}d(\Psi dz_m)\\
   =&\oint_Adz_m(\Psi(z_m+\tau)-\Psi(z_m))-\oint_Bdz_m(\Psi(z_m+1)-\Psi(z_m))+\sum_{1\leq k\leq m-1}\Res_{z_m=z_k}\Psi(z_m)\\
   =&\oint_Adz_m\left(F_0+\frac{1}{2\sqrt{-1}\operatorname{Im}\tau}F_1+\sum^{n-1}_{l=1} \frac{\partial_{z_m}^lF_{l}+\frac{l+1}{2\sqrt{-1}\operatorname{Im}\tau}\partial_{z_m}^lF_{l+1}}{l+1}+\frac{\partial^n_{z_m}F_n}{n+1}\right) \\
    &+\sum_{1\leq k\leq m-1}\operatorname{Res}_{z_m=z_k}\left(\frac{\operatorname{Im}z_m}{\operatorname{Im}\tau}F_0+\sum^n_{l=1}\partial_{z_m}\bigg(\frac{1}{l+1}\bigg(\frac{\operatorname{Im}z_m}{\operatorname{Im}\tau}\bigg)^{l+1}\partial^{l-1}_{z_m}F_l\bigg)\right)\\
   =&\oint_Adz_m\left(F_0+\frac{1}{2\sqrt{-1}\operatorname{Im}\tau}F_1\right)+\sum_{1\leq k\leq m-1}\operatorname{Res}_{z_m=z_k}\frac{\operatorname{Im}z_m}{\operatorname{Im}\tau}F_0. \qedhere
\end{align*}
\end{proof}

Now let us start with the proof of Theorem \ref{ModifiedAcycleElliptic}.

\begin{proof}[Proof of Theorem \ref{ModifiedAcycleElliptic}] We prove it by induction on $|N|$. When $|N|=0$, there is nothing to prove. Now we assume $\widehat{A}_N[a]$ is almost-meromorphic and elliptic for any $M$ and $a\in \mathcal{A}^{\otimes(m+n)},m=|M|.$
\vskip 0.1cm
We decompose $\widehat{A}_{N}[v]$ as the above lemma for $F$,
$$
\widehat{A}_{N}[v]=\displaystyle \sum _{ \substack{I\subset N,\\
\phi:I\rightarrow M'}}D^{\phi}_{I\rightarrow M'}A_{{I}^c} [B^{\phi}_{{I\rightarrow M'}}v]+\displaystyle \sum _{ \substack{I\cup I_1\subset N,I_1\neq\emptyset\\
\phi:I\rightarrow M'}}D_{I_1\rightarrow m}D^{\phi}_{I\rightarrow M'}A_{{(I\cup I_1)}^c} [B_{I_1\rightarrow m}B^{\phi}_{{I\rightarrow M}}v],
$$
and use Lemma \ref{RegIntegralF}, we get
\begin{align*}
    E_m\widehat{A}_N[a] =&\oint_Adz_m\displaystyle \sum _{ \substack{I\subset N,\\
\phi:I\rightarrow M'}}D^{\phi}_{I\rightarrow M'}A_{{I}^c} [B^{\phi}_{{I\rightarrow M'}}a]+\underbrace{\oint_Adz_m\displaystyle \sum _{ \substack{I\cup \{l\}\subset N\\
\phi:I\rightarrow M'}}D^{\phi}_{I\rightarrow M'}A_{{(I\cup \{l\})}^c} [B^{\phi}_{{I\rightarrow M'}}B_{l\rightarrow m}a]}_{\Circled{1}}\\
&+\sum_{k\in M'}\operatorname{Res}_{z_m=z_k}\frac{\operatorname{Im}z_m}{\operatorname{Im}\tau}\displaystyle \sum _{ \substack{I\subset N,\\
\phi:I\rightarrow M'}}D^{\phi}_{I\rightarrow M'}A_{{I}^c} [B^{\phi}_{{I\rightarrow M'}}a].
\end{align*}
To compute the residue term $\operatorname{Res}_{z_m=z_k}$, for each $k\in M'$, we decompose again
\begin{align*}
   & \displaystyle \sum _{ \substack{I\subset N,\\
\phi:I\rightarrow M'}}D^{\phi}_{I\rightarrow M'}A_{{I}^c} [B^{\phi}_{{I\rightarrow M'}}a] \\
=&\underbrace{\displaystyle \sum _{ \substack{I\subset N,\\
\phi:I\rightarrow M'_k}}D^{\phi}_{I\rightarrow M'_k}A_{{I}^c} [B^{\phi}_{{I\rightarrow M'_k}}a]}_{\Circled{2}}+\underbrace{\displaystyle \sum _{ \substack{I\cup I_1\subset N, I_1\neq \emptyset\\
\phi:I\rightarrow M'_k}}D_{I_1\rightarrow k}D^{\phi}_{I\rightarrow M'_k}A_{{I}^c} [B_{I_1\rightarrow k}B^{\phi}_{{I\rightarrow M'_k}}a]}_{\Circled{3}},
\end{align*}

The three terms $\Circled{1}$, $\Circled{2}$, and $\Circled{3}$ are computed by the following lemmas \ref{B=C}, \ref{SingleCollide} and \ref{MultipleCollide}, respectively. Then we get
\begin{align*}
    E_m\widehat{A}_N[a] =&\oint_Adz_m\displaystyle \sum _{ \substack{I\subset N,\\
\phi:I\rightarrow M'}}D^{\phi}_{I\rightarrow M'}A_{{I}^c} [B^{\phi}_{{I\rightarrow M'}}a]+ \underbrace{\frac{1}{2\sqrt{-1}\operatorname{Im}\tau}\displaystyle \sum _{ \substack{I\cup \{l\}\subset N\\
\phi:I\rightarrow M'}}D^{\phi}_{I\rightarrow M'}A_{{I}^c} [B^{\phi}_{{I\rightarrow M'}}C_{m\rightarrow l}a]}_{\Circled{1}}\\
   & +\underbrace{\sum_{k\in M'}\displaystyle \sum _{ \substack{I\subset N,\\
\phi:I\rightarrow M'_k}}D_{1,k}D^{\phi}_{I\rightarrow M'_k}A_{{I}^c} [B_{m\rightarrow k}B^{\phi}_{{I\rightarrow M'_k}}a]+\frac{1}{2\sqrt{-1}\operatorname{Im}\tau}\sum_{k\in M'}\displaystyle \sum _{ \substack{I\subset N,\\
\phi:I\rightarrow M'_k}}D^{\phi}_{I\rightarrow M'_k}A_{{I}^c} [B^{\phi}_{{I\rightarrow M'_k}}C_{m\rightarrow k}a]}_{\Circled{2}}\\
&\left.
\begin{aligned}
&+\sum_{k\in M'}\displaystyle \sum _{ \substack{I\cup I_1\subset N,I_1\neq \emptyset\\
\phi:I\rightarrow M'_k}}D_{I_1\cup \{m\}\rightarrow k}D^{\phi}_{I\rightarrow M'_k}A_{{I}^c} [B_{I_1\cup\{m\}\rightarrow k}B^{\phi}_{{I\rightarrow M'_k}}a]\\
&+\frac{1}{2\sqrt{-1}\operatorname{Im}\tau}\sum_{k\in M'}\displaystyle \sum _{ \substack{I\cup I_1\subset N,I_1\neq \emptyset\\
\phi:I\rightarrow M'_k}}D_{I_1 \rightarrow k}D^{\phi}_{I\rightarrow M'_k}A_{{I}^c} [B_{I_1\rightarrow k}B^{\phi}_{{I\rightarrow M'_k}}C_{m\rightarrow k}a]\\
&+\frac{1}{2\sqrt{-1}\operatorname{Im}\tau}\sum_{k\in M'}\displaystyle \sum _{ \substack{I\cup I_1\subset N, l\in I_1\\
\phi:I\rightarrow M'_k}}D_{I_1 \rightarrow k}D^{\phi}_{I\rightarrow M'_k}A_{{I}^c} [B_{I_1\rightarrow k}B^{\phi}_{{I\rightarrow M'_k}}C_{m\rightarrow l}a]\\
\end{aligned}
\right\}
~\Circled{3}\\
=~&\widehat{A}_{N\cup\{m\}}[a]+\frac{1}{2\sqrt{-1}\operatorname{Im}\tau}\sum_{0\leq j\leq m+n,j\neq m}\widehat{A}_{N}[C_{m\rightarrow j}a].
\end{align*}
The last equality is obtained by combining the $C$-terms and the remaining ones respectively.
Since the surface integral of an elliptic function is still elliptic and preserves the axioms (1)-(5), by the induction hypothesis and the above formula, $\widehat{A}_{N\cup\{m\}}$ is elliptic and obviously almost-meromorphic, satisfies the axioms (1)-(5).
\end{proof}

\begin{lem}\label{B=C}
We have
$$
\oint_Adz_m\displaystyle \sum _{ \substack{I\cup \{l\}\subset N\\
\phi:I\rightarrow M'}}D^{\phi}_{I\rightarrow M'}A_{{(I\cup \{l\})}^c} [B^{\phi}_{{I\rightarrow M'}}B_{l\rightarrow m}a]=\displaystyle \sum _{ \substack{I\cup \{l\}\subset N\\
\phi:I\rightarrow M'}}D^{\phi}_{I\rightarrow M'}A_{{I}^c} [B^{\phi}_{{I\rightarrow M'}}C_{m\rightarrow l}a]
$$
\end{lem}

\begin{proof} This lemma follows from Lemma \ref{BCS} and the fact that both $\oint_Adz_m$ and $\oint_Adz_l$ commute with the operations $D^{\phi}_{I\rightarrow M'}$ and $B^{\phi}_{I\rightarrow M'}$, since $l\notin I$ and $m\notin M'$.
\end{proof}

It can be explained by the diagram as follows.

  \begin{figure}[H]
    \centering

\tikzset{every picture/.style={line width=0.75pt}} 


        \caption{Lemma \ref{B=C}}
  \end{figure}

\begin{lem}\label{SingleCollide}
We have
\begin{align*}
   & \operatorname{Res}_{z_m=z_k}\frac{\operatorname{Im}z_m}{\operatorname{Im}\tau}\displaystyle \sum _{ \substack{I\subset N,\\
\phi:I\rightarrow M'_k}}D^{\phi}_{I\rightarrow M'_k}A_{{I}^c} [B^{\phi}_{{I\rightarrow M'_k}}a] \\
   & =\displaystyle \sum _{ \substack{I\subset N,\\
\phi:I\rightarrow M'_k}}D_{1,k}D^{\phi}_{I\rightarrow M'_k}A_{{I}^c} [B_{m\rightarrow k}B^{\phi}_{{I\rightarrow M'_k}}v]+\frac{1}{2\sqrt{-1}\operatorname{Im}\tau}\displaystyle \sum _{ \substack{I\subset N,\\
\phi:I\rightarrow M'_k}}D^{\phi}_{I\rightarrow M'_k}A_{{I}^c} [B^{\phi}_{{I\rightarrow M'_k}}C_{m\rightarrow k}a].
\end{align*}
\end{lem}
\begin{proof}
Note that the operation $\operatorname{Res}_{z_m=z_k}\frac{\operatorname{Im}z_m}{\operatorname{Im}\tau}$ commutes with the operation $\displaystyle \sum _{ \substack{I\subset N,\\
\phi:I\rightarrow M'_k}}D^{\phi}_{I\rightarrow M'_k}A_{{I}^c} $, so we only need to compute
$$
 \operatorname{Res}_{z_m=z_k}\frac{\operatorname{Im}z_m}{\operatorname{Im}\tau}[B^{\phi}_{{I\rightarrow M'_k}}a], \quad \text{for } I\subset N, \phi:I\rightarrow M'_k.
$$
It is done by Equation \eqref{modfiedBCD}.
\end{proof}

Similarly, it can be explained by the diagram as follows.
  \begin{figure}[H]
    \centering

\tikzset{every picture/.style={line width=0.75pt}} 


    \caption{Lemma \ref{SingleCollide}}
  \end{figure}

\begin{lem}\label{MultipleCollide}
  We have\begin{align*}
&\operatorname{Res}_{z_m=z_k}\frac{\operatorname{Im}z_m}{\operatorname{Im}\tau}\displaystyle \sum _{ \substack{I\cup I_1\subset N,\\
\phi:I\rightarrow M'_k}}D_{I_1\rightarrow k}D^{\phi}_{I\rightarrow M'_k}A_{{I}^c} [B_{I_1\rightarrow k}B^{\phi}_{{I\rightarrow M'_k}}a] \\
=~ & \displaystyle \sum _{ \substack{I\cup I_1\subset N,\\
\phi:I\rightarrow M'_k}}D_{I_1\cup \{m\}\rightarrow k}D^{\phi}_{I\rightarrow M'_k}A_{{I}^c} [B_{I_1\cup\{m\}\rightarrow k}B^{\phi}_{{I\rightarrow M'_k}}a]\\
&+\frac{1}{2\sqrt{-1}\operatorname{Im}\tau}\displaystyle \sum _{ \substack{I\cup I_1\subset N,\\
\phi:I\rightarrow M'_k}}D_{I_1 \rightarrow k}D^{\phi}_{I\rightarrow M'_k}A_{{I}^c} [B_{I_1\rightarrow k}B^{\phi}_{{I\rightarrow M'_k}}C_{m\rightarrow k}a]\\
&+\frac{1}{2\sqrt{-1}\operatorname{Im}\tau}\displaystyle \sum _{ \substack{I\cup I_1\subset N, l\in I_1\\
\phi:I\rightarrow M'_k}}D_{I_1 \rightarrow k}D^{\phi}_{I\rightarrow M'_k}A_{{I}^c} [B_{I_1\rightarrow k}B^{\phi}_{{I\rightarrow M'_k}}C_{m\rightarrow l}a].
\end{align*}

\end{lem}
\begin{proof}
For any $I\cup I_1\subset N,\phi:I\rightarrow M'_k$, the operation $\operatorname{Res}_{z_m=z_k}\frac{\operatorname{Im}z_m}{\operatorname{Im}\tau}D_{I_1\rightarrow k}$ commute with $D^{\phi}_{I\rightarrow M'_k}A_{{I}^c}$. Thus we only need to consider
$$
   \operatorname{Res}_{z_m=z_k}\frac{\operatorname{Im}z_m}{\operatorname{Im}\tau}D_{I_1\rightarrow k} [B_{I_1\rightarrow k}B^{\phi}_{{I\rightarrow M'_k}}a].
$$
By a direct computation, we have
\begin{align*}
  &\operatorname{Res}_{z_m=z_k}\frac{\operatorname{Im}z_m}{\operatorname{Im}\tau} D_{I_1\rightarrow k}[B_{I_1\rightarrow k}B^{\phi}_{{I\rightarrow M'_k}}a]\\
=~& D_{I_1\rightarrow k}\operatorname{Res}_{z_m=z_k}\frac{\operatorname{Im}z_m}{\operatorname{Im}\tau}[B_{I_1\rightarrow k}B^{\phi}_{{I\rightarrow M'_k}}a]\\
\overset{(1)}=~& D_{I_1\rightarrow k}\p_{z_k}\left(\frac{\operatorname{Im}z_k}{\operatorname{Im}\tau}[B_{m\rightarrow k}B_{I_1\rightarrow k}B^{\phi}_{{I\rightarrow M'_k}}a]\right)+\frac{1}{2\sqrt{-1}\operatorname{Im}\tau}D_{I_1\rightarrow k}[C_{m\rightarrow k}B_{I_1\rightarrow k}B^{\phi}_{{I\rightarrow M'_k}}a]\\
\overset{(2)}=~&D_{I_1\cup \{m\}\rightarrow k} [B_{I_1\cup\{m\}\rightarrow k}B^{\phi}_{{I\rightarrow M'_k}}a]+\frac{1}{2\sqrt{-1}\operatorname{Im}\tau}D_{I_1\rightarrow k}|I_1|[B_{m\rightarrow k}B_{I_1\rightarrow k}B^{\phi}_{{I\rightarrow M'_k}}a]\\
&+\frac{1}{2\sqrt{-1}\operatorname{Im}\tau}D_{I_1\rightarrow k}[C_{m\rightarrow k}B_{I_1\rightarrow k}B^{\phi}_{{I\rightarrow M'_k}}a]\\
\overset{(3)}=~&D_{I_1\cup \{m\}\rightarrow k} [B_{I_1\cup\{m\}\rightarrow k}B^{\phi}_{{I\rightarrow M'_k}}a]+\frac{1}{2\sqrt{-1}\operatorname{Im}\tau}D_{I_1\rightarrow k} [B_{I_1\rightarrow k}B^{\phi}_{{I\rightarrow M'_k}}C_{m\rightarrow k}a]\\
&+\frac{1}{2\sqrt{-1}\operatorname{Im}\tau}\sum_{l\in I_1}D_{I_1\rightarrow k} [B_{I_1\rightarrow k}B^{\phi}_{{I\rightarrow M'_k}}C_{m\rightarrow l}a],
\end{align*}
where (1) holds by Equation \eqref{modfiedBCD}; (2) holds by definition of $D_{I_1\rightarrow k}$ and
$$\left[\p_{z_k}^{|I_1|},\frac{\operatorname{Im}z_k}{\operatorname{Im}\tau}\right]=\frac{|I_1|}{2\sqrt{-1}\operatorname{Im}\tau}\p_{z_k}^{|I_1|-1};$$
and (3) holds by Lemma \ref{Useful}.
\end{proof}

Similarly, it can be explained by the diagram as follows.
\begin{figure}[H]
  \centering

\tikzset{every picture/.style={line width=0.75pt}} 


  \caption{Lemma \ref{MultipleCollide}}
\end{figure}

Now we turn to prove Theorem \ref{Contactequation}.
\begin{lem}\label{VanishTotal}
  We have
  $$
  E_M\widehat{A}_N(\partial_{z_k}[a])=0, \quad\forall~ k\in M, ~a\in\mathcal{A}^{\otimes(m+n)}.
  $$
\end{lem}
\begin{proof}
We prove this by induction on $|N|$. If $|N|=0,$ then
$$
E_M\partial_{z_k}[a]=E_{M-\{k\}}E_k\partial_{z_k}[a]=0.
$$
Now suppose $|N|>0$, we pick $l\in N$ and use Proposition \ref{ModifiedAcycleElliptic}
\begin{align*}
  E_M\widehat{A}_N\partial_{z_k}[a] &=E_M(E_l\widehat{A}_{N-\{l\}}[T_ka]-\frac{1}{2\sqrt{-1}\operatorname{Im}\tau}\sum_{i\neq l}\widehat{A}_{N-\{l\}}[C_{l\rightarrow i}T_ka]) .
\end{align*}
Now note that
\begin{itemize}
  \item For $i\neq k$, $[C_{l\rightarrow i}, T_k]=0$;
  \item For $i=k$, the $k$-th element of $C_{l\rightarrow k}T_ka$ is
  \begin{align*}
  C(a_l,Ta_k)&=T^{-1}((Ta_k)_{(0)}a_l)+Tb\quad\text{for some }b\in\mathcal{A}\\
           &=Tb.  \quad (\text{Proposition }\ref{immediateconseq}, (Ta_k)_{(0)}=0)
  \end{align*}
\end{itemize}
By the induction hypothesis
$$
E_M\frac{1}{2\sqrt{-1}\operatorname{Im}\tau}\sum_{i\neq l}\widehat{A}_{N-\{l\}}[C_{l\rightarrow i}T_ka]=0,
$$
$$
E_ME_l\widehat{A}_{N-\{l\}}[T_ka]=0,
$$
combining the above identities, we obtain the lemma.
\end{proof}
\begin{proof}[Proof of Theorem \ref{Contactequation}] It follows from the above lemma and $C(v_i,v_j)-\sum_{k\in \mathbb{N}_+}c^k_{ij}v_k\in \operatorname{Im}(T)$.
\end{proof}

\begin{appendices}

\section{Chiral boson}\label{proofnormal}

In Example \ref{Walg}, we show that the $W_{\infty}$ algebra as a subalgebra of the free boson vertex algebra satisfies Dijkgraaf's integrable condition. To define various types of genus one partition function using two-step constructions, we need to construct a genus one conformal block first. For the chiral boson, we use normal ordered Feynman graphs (Subsection \ref{normalorderedFG}) to define the desired functions (Definition \ref{defnnormalorderedFG} and Proposition \ref{normalcorrelation}). In this section, we will give the detailed proof of Proposition \ref{normalcorrelation}.

\subsection{Proof of Axiom (1)-(5)}
\begin{proof}[Proof of Axiom (1)-(5)]
The axioms (1) and (2), i.e., the linearity for each $a_i\in V$ and the invariance of permutation, are obvious by definition.

(3)~$S((|0\rangle,z),(a_1,z_1),\ldots,(a_n,z_n),\tau)=S((a_1,z_1),\ldots,(a_n,z_n),\tau)$:

It follows from $\L(z):=Y(|0\rangle,z)=\Id$, which corresponds to an isolated vertex without tails and self-loops in the graph. So, the input of $\L(z)$ does not change the Feynmann graph function $\Phi_{\Gamma}$.

(4)~$S((L_{-1}a_1,z_1),\ldots,(a_n,z_n),\tau)=\frac{d}{dz_1}S((a_1,z_1),\ldots,(a_n,z_n),\tau)$:
\begin{align*}
&S((L_{-1}a_1,z_1),\ldots,(a_n,z_n),\tau)\\
=~&\eta(\tau)^{-1}\langle :Y(L_{-1}a_1, z_1):,:Y(a_2,z_2):,\cdots,:Y(a_n,z_n):;\tau\rangle\\
=~&\eta(\tau)^{-1}\langle :\frac{d}{dz_1}Y(a_1, z_1):,:Y(a_2,z_2):,\cdots,:Y(a_n,z_n):;\tau\rangle\\
\overset{*}=~&\eta(\tau)^{-1}\frac{d}{dz_1}\langle :Y(a_1, z_1):,:Y(a_2,z_2):,\cdots,:Y(a_n,z_n):;\tau\rangle\\
=~&\frac{d}{dz_1}S((a_1,z_1),\ldots,(a_n,z_n),\tau).
\end{align*}
Here the equality $*$ follows from the fact that $\frac{d}{dz}$ commutes with the two types of propagators, more explicitly,
  \begin{align*} &\frac{d}{dz}\left(\sum_{k,l\geq0}\partial_{z}^k\partial_{w}^lP(z,w)\frac{\p}{\p(\p^{k+1}\phi(z))}\frac{\p}{\p(\p^{l+1}\phi(w))}\p^i\phi(z)\p^j\phi(w)\right)\\
     =~& \left(\sum_{k,l\geq0}\partial_{z}^k\partial_{w}^lP(z,w)\frac{\p}{\p(\p^{k+1}\phi(z))}\frac{\p}{\p(\p^{l+1}\phi(w))}\right)\p^{i+1}\phi(z)\p^j\phi(w).\\
\text{and }~&\\
  &\frac{d}{dz}\left(\sum_{k,l\geq0}\partial_{z}^k\partial_{w}^lQ(z,w)\bigg|_{w=z}\frac{\p}{\p(\p^{k+1}\phi(z))}\frac{\p}{\p(\p^{l+1}\phi(z))}\p^i\phi(z)\p^j\phi(z)\right)=0\\
  &\left(\sum_{k,l\geq0}\partial_{z}^k\partial_{w}^lQ(z,w)\bigg|_{w=z}\frac{\p}{\p(\p^{k+1}\phi(z))}\frac{\p}{\p(\p^{l+1}\phi(z))}\right)\bigg(\p^{i+1}\phi(z)\p^j\phi(z)+\p^i\phi(z)\p^{j+1}\phi(z)\bigg)=0.
  \end{align*}

The axiom (5) is given by the following lemma \ref{graphOPE}. \qedhere

\end{proof}

\begin{lem}\label{graphOPE} When $z_1$ is close to $z_2$, we have
\begin{equation}\label{grOPE}
\langle :\L_1(z_1):,:\L_2(z_2):,\cdots,:\L_n(z_n):\rangle=\bigg\langle:\sum_{n\in\Z}\frac{({\L_1}_{(n)}\cdot\L_2)(z_2)}{(z_1-z_2)^{n+1}}:,\cdots,:\L_n(z_n):\bigg\rangle.
\end{equation}
Where at the right hand side, $\sum_{n\in\Z}\frac{({\L_1}_{(n)}\L_2)(z_2)}{(z_1-z_2)^{n+1}}$ is the OPE of $\L_1(z_1)\L_2(z_2)$.
\end{lem}

\begin{proof} To prove the equality, we only need to focus on the $z_1$-, $z_2$-local parts of all the Feynman graphs, such as Figure 2.
\begin{figure}[H]
	\centering
	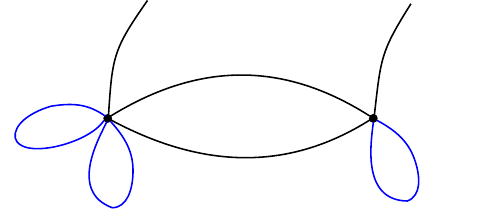
	\caption[]{The local part of a graph.}
	\label{Fig:Q1}
\end{figure}

Note that the singular part of the OPE comes from the singular part of the propagator. More explicitly, we have
$$\L_1(z_1)\L_2(z_2)=\exp\bigg(\sum_{k,l\geq0}\partial_{z_1}^k\partial_{z_2}^lR(z_1,z_2)\frac{\p}{\p(\p^{k+1}\phi(z_1))}\frac{\p}{\p(\p^{l+1}\phi(z_2))}\bigg):\L_1(z_1)\L_2(z_2):,$$
then we do the Taylor expansion at $z_2$ and get the OPE.

To prove \eqref{grOPE}, it suffices to compare the contributions of $Q(z_1,z_1)$, $Q(z_1,z_2)$, $Q(z_2,z_2)$, $P(z_1,z_l)$, $P(z_2,z_l)$ on the left hand side and those of $Q(z_2,z_2)$, $P(z_2,z_l)$ on the right hand side of \eqref{grOPE}. In fact, we only need to consider the following three cases.

The first case is given by the edge connecting $\p^i\phi(z_1)$ to $\p^k\phi(z_2)$:
\vskip 0.1cm
\begin{figure}[H]
	\centering
	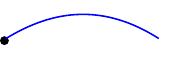
	\caption[]{Contribution of $Q(z_1,z_2)$.}
	\label{Fig:Q1}
\end{figure}
\vskip -0.7cm
\begin{equation*}
\begin{aligned}
\LHS~&=~\p_{z_1}^{i-1}\p_{z_2}^{k-1}Q(z_1,z_2)\\
    &=~(-1)^{k-1}\sum_{m\geq i+k-2}\frac{m!}{(m-i-k+2)!}Q_m(z_1-z_2)^{m-i-k+2}\\
    &=~(-1)^{k-1}\sum_{n\geq0}\frac{1}{n!}(z_1-z_2)^n(i+n+k-2)!Q_{i+n+k-2}\\
    &=~\sum_{n\geq0}\p_{z_1}^{i+n-1}\p_{z_2}^{k-1}Q(z_1,z_2)\big|_{z_1=z_2}\frac{1}{n!}(z_1-z_2)^n=\RHS.
\end{aligned}
\end{equation*}

The second case is given by the self-loop connecting $\p^i\phi(z_1)$ to $\p^j\phi(z_1)$:
\vskip 0.2cm
\begin{figure}[H]
	\centering
	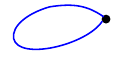
	\caption[]{Contribution of $Q(z_1,z_1)$.}
	\label{Fig:Q1}
\end{figure}

\vskip -0.8cm
\begin{align*}
\LHS=~&\p_{z_1}^{i-1}\p_w^{j-1}Q(z_1,w)\bigg|_{w=z_1}\\
=~&\sum_{m\geq0}\left[\p_{z_2}^{i+m-1}\p_w^{j-1}Q(z_2,w)\frac{(z_1-z_2)^m}{m!}\right]\bigg|_{w=z_1}\\
=~&\sum_{m,n\geq0}\left[\p_{z_2}^{i+m-1}\p_w^{j+n-1}Q(z_2,w)\big|_{w=z_2}\frac{(w-z_2)^n}{n!}\frac{(z_1-z_2)^m}{m!}\right]\bigg|_{w=z_1}\\
=~&\sum_{m,n\geq0}\p_{z_2}^{i+m-1}\p_w^{j+n-1}Q(z_2,w)\bigg|_{w=z_2}\frac{(z_1-z_2)^{m+n}}{m!n!}=\RHS.
\end{align*}
That is, both sides are equal to $(-1)^{j-1}(i+j-2)!Q_{i+j-2}$.
\vskip 0.1cm

The third case is given by the edge connecting $\p^i\phi(z_1)$ to $\p^j\phi(z_l)$ $(l\neq2)$:
\begin{figure}[H]
	\centering
	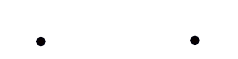
	\caption[]{Contribution of $P(z_1,z_l)$.}
	\label{Fig:Q1}
\end{figure}
\vskip -0.8cm
$$\LHS~=~\p_{z_1}^{i-1}\p_{z_l}^{j-1}P(z_1-z_l)=~\sum_{m\geq0}\frac{1}{m!}\p_{z_2}^{i-1+m}\p_{z_l}^{j-1}P(z_2-z_l)(z_1-z_2)^m=\RHS.\qedhere$$
\end{proof}

\subsection{Proof of Axiom (6)}
The proof of (6) will depend on the properties of the Weierstrass $\wp$ function. In particular, the product formulas (Proposition \ref{prodformula}) of the propagator $P(z-w;\tau)$ play an important role. First, let us recall the relevant functions including the modified Weierstrass $\wp$ function (i.e. the propagator), the modified Weierstrass zeta function and the Jacobi theta function, and discuss their relations. For more details, see \cite{Lang}.

The modified Weierstrass elliptic function $P(z;\tau)$ is defined by
$$P(z;\tau)=\wp(z;\tau)+\frac{\pi^2}{3}E_2(\tau).$$
The corresponding modified Weierstrass zeta function is defined by
$$\zeta(z;\tau)=~\frac{1}{z}+\sum_{(m,n)\in(\mathbb{Z}+\mathbb{Z})\setminus\{(0,0)\}}\left(\frac{1}{z-m-n\tau}+\frac{1}{m+n\tau}+\frac{z}{(m+n\tau)^2}\right)-\frac{\pi^2}{3}zE_2(\tau).$$
The following lemma is obvious.
\begin{lem} The functions $P(z;\tau)$ and $\zeta(z;\tau)$ satisfy the following properties
\begin{enumerate}
  \item $P(z;\tau)$ is even and elliptic in $z$;
  \item $\zeta(z;\tau)$ is odd and quasi-elliptic in $z$:
  $$\zeta(z+1;\tau)=\zeta(z;\tau),\quad \zeta(z+\tau;\tau)=\zeta(z;\tau)-2\pi i.$$
  Clearly, $P(z;\tau)=-\p_z\zeta(z;\tau).$
\end{enumerate}
\end{lem}

Another important function is the Jacobi theta function $\theta(z;\tau)$, which is defined to be
$$\theta(z;\tau)=\sum_{n=-\infty}^{\infty}\exp(\pi in^2\tau+2\pi i n z).$$
It is 1-periodic in $z$ and $\tau$-quasi-periodic in $z$. More explicitly, for $m,n\in\mathbb{Z}$,
$$\theta(z+m+n\tau;\tau)=\exp(-\pi in^2\tau-2\pi i nz)\theta(z;\tau).$$

It is known that the Jacobi theta function $\theta(z;\tau)$ satisfies the heat equation
\begin{equation}\label{thetaheat}
\p_{\tau}\theta(z;\tau)=-\frac{i}{4\pi}\p_z^2\theta(z;\tau).
\end{equation}

The Jacobi theta function $\theta(z;\tau)$ is related to $\zeta(z;\tau)$ and $P(z;\tau)$ as follows. First, the Jacobi theta function defined above is also considered along with three auxiliary theta functions. Here we only consider
$$\theta_1(z;\tau)=-\exp\bigg(\frac{1}{4}\pi i\tau+\pi i(z+\frac{1}{2})\bigg)\theta(z+\frac{1}{2}+\frac{1}{2}\tau;\tau).$$
Then
$$\zeta(z;\tau)=-\p_z\log\theta_1(z;\tau),\quad P(z;\tau)=\p_z^2\log\theta_1(z;\tau).$$

Using the heat equation \eqref{thetaheat} for $\theta(z;\tau)$ and the above relations, we get
\begin{lem} The following identity holds
\begin{equation*}
2\pi i\frac{d}{d\tau}\zeta(z)=-\zeta(z)P(z)-\frac{1}{2}P'(z).
\end{equation*}
Moreover, taking derivatives with respect to $z$, the following identity holds
\begin{equation}\label{PDE}
2\pi i\frac{d}{d\tau}P=\zeta P'-P^2+\frac{1}{2}P''.
\end{equation}
\end{lem}

Furthermore, we give the following additional formulas for the Weierstrass $\wp$ functions and the modified zeta functions, which lead to the product formulas \ref{prodformula} for the propagators with some efforts.

\begin{lem}[additional formula]\label{wpeq} The Weierstrass $\wp$ function satisfies
$$\wp(z+w)=\frac{1}{4}\left\{\frac{\wp'(z)-\wp'(w)}{\wp(z)-\wp(w)}\right\}^2-\wp(z)-\wp(w).$$
\end{lem}

\begin{lem}
\begin{equation}\label{Peq}
\frac{P'(z)-P'(w)}{P(z)-P(w)}=2\zeta(z+w)-2\zeta(z)-2\zeta(w).
\end{equation}
\end{lem}

\begin{proof} This lemma follows from the results in Chapter 18 of \cite{Lang}. Recall that the Weierstrass sigma function $\sigma$ is defined by
$$\sigma(z)=z\prod_{\omega\in(\mathbb{Z}+\mathbb{Z}\tau)\setminus\{(0,0)\}}\left(1-\frac{z}{\omega}\right)e^{z/\omega+\frac{1}{2}(z/\omega)^2}.$$
It is known that $\sigma(z)$ is related to the modified Weierstrass zeta function as follows
$$\zeta(z;\tau)+\frac{\pi^2}{3}zE_2(\tau)=\frac{\sigma'(z)}{\sigma(z)}$$
Using the identity
$$P(z)-P(w)=\wp(z)-\wp(w)=-\frac{\sigma(z+w)\sigma(z-w)}{\sigma^2(z)\sigma^2(w)}.$$
Taking the logarithm of this identity and then taking derivatives with respect to $z$, we get
$$\frac{P'(z)}{P(z)-P(w)}=\zeta(z+w)+\zeta(z-w)-2\zeta(z).$$
Repeating the computation with respect to $w$, we get the desired equation.
\end{proof}

\begin{prop}[Product formulas]\label{prodformula} (1) For three distinct points $u$, $v$ and $w$, we have the product formula
\begin{align*}
P(u-w)P(w-v)=~&2\pi i\frac{d}{d\tau}P(u-v)+\big(P(w-u)+P(w-v)\big)P(u-v)\\
                 &+\big(\zeta(w-u)-\zeta(w-v))\frac{d}{du}P(u-v).
\end{align*}

(2) For two distinct points $u$, $w$, the following identity holds
$$P(w-u)P(w-u)=2\pi i\frac{d}{d\tau}\frac{\pi^2}{3}E_2(\tau)+\frac{2\pi^2}{3}E_2(\tau)P(w-u)+\frac{1}{6}\frac{d^2}{dw^2}P(w-u).$$
\end{prop}

\begin{proof}
(1) First, by Lemma \ref{wpeq}, $P(u-v)=\wp(u-v)+\frac{\pi^2}{3}E_2(\tau)$ satisfies
\begin{equation}\label{P(u-v)}
P(u-v)=\frac{1}{4}\left\{\frac{P'(u-w)-P'(w-v)}{P(u-w)-P(w-v)}\right\}^2-P(u-w)-P(w-v)+\frac{\pi^2}{3}E_2(\tau).
\end{equation}
Substituting Equation \eqref{Peq} into Equation $\eqref{P(u-v)}$, and then taking derivatives with respect to $u$ and $v$, we get
\begin{align*}
-P''(u-v)=&-2P(u-v)^2+2\zeta(u-v)P'(u-v)+2\big(P(u-w)+P(w-v)\big)P(u-v)\\
          &-2\big(\zeta(u-w)+\zeta(w-v)\big)P'(u-v)-2P(u-w)P(w-v).
\end{align*}
Finally, using the differential equation \eqref{PDE}, we get the first product formula.
\vskip 0.1cm

(2) By a direct computation,
\begin{align*}
P(w-u)P(w-u)=~&(\wp(w-u)+\frac{\pi^2}{3}E_2)^2\\
                =~&\wp(w-u)^2+\frac{2\pi^2}{3}E_2\wp(w-u)+\left(\frac{\pi^2}{3}E_2\right)^2\\
                \overset{(1)}=~&\frac{1}{6}\frac{d^2}{dw^2}\wp(w-u)+\frac{1}{9}\pi^4E_4+\frac{1}{9}\pi^4E_2^2+\frac{2}{3}\pi^2E_2\wp(w-u)\\
                \overset{(2)}=~&2\pi i\frac{d}{d\tau}\left(\frac{\pi^2}{3}E_2\right)+\frac{2\pi^2}{3}E_2P(w-u)+\frac{1}{6}\frac{d^2}{dw^2}\wp(w-u).
\end{align*}
Here in the equality (1), we use the defining equation of the Weierstrass $\wp$ function
$$\wp'(z)^2=4\wp^3-g_2\wp-g_3,\quad g_2(\tau)=\frac{4}{3}\pi^4E_4(\tau),$$
and in the equality (2), we use the Ramanujan identity
$$\frac{1}{2\pi i}\p_{\tau}E_2=\frac{1}{12}(E_2^2-E_4).\qedhere$$

\end{proof}

\begin{lem}\label{Axiom(6)} Let $a_i=\phi_{-1}1, i=1,\ldots, n$ be the primary vectors for the Virasoro algebra, then
\begin{align*}
S=~&S((\omega,w_1),\ldots,(\omega,w_m),(a_1,z_1),\ldots,(a_n,z_n),\tau)\\
 =~&\eta(\tau)^{-1}\langle:T(w_1):,\ldots,:T(w_m):,:\p\phi(z_1):,\ldots,:\p\phi(z_n):;\tau\rangle
\end{align*}
satisfies the axiom (6):
\begin{align*}
&S((\omega,w),(\omega,w_1),\ldots,(\omega,w_m),(a_1,z_1),\ldots,(a_n,z_n),\tau)\\
=~&2\pi i\frac{d}{d\tau}S+\sum_{k=1}^n(\wp_1(w-z_k,\tau)-(w-z_k)\frac{\pi^2}{3}E_2(\tau))\frac{d}{dz_k}S\\
&+\sum_{k=1}^m(\wp_1(w-w_k,\tau)-(w-w_k)\frac{\pi^2}{3}E_2(\tau))\frac{d}{dw_k}S\\
&+\sum_{k=1}^n(\wp_2(w-z_k,\tau)+\frac{\pi^2}{3}E_2(\tau))S\\
&+\sum_{k=1}^m(\wp_2(w-w_k,\tau)+\frac{\pi^2}{3}E_2(\tau))2S\\
&+\frac{1}{2}\sum_{k=1}^m\wp_4(w-w_k,\tau)S((\omega,w_1),\ldots,\widehat{(\omega,w_k)},\ldots,(a_n,z_n),\tau).
\end{align*}
\end{lem}

\begin{proof} We consider the graph change by adding the new vertex $T(w)=\frac{1}{2}:(\p\phi(w))^2:$. The changes include the following three cases:

\noindent\textbf{Case 1.} $T(w)=\frac{1}{2}:(\p\phi(w))^2:$ connects to itself, i.e., for each original graph, adding an isolated self-loop:
\begin{figure}[H]
	\centering
	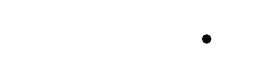
	\caption[]{CASE 1.}
	\label{Fig:Q1}
\end{figure}

Then we get a contribution $\frac{\pi^2}{3}E_2S$.

\noindent\textbf{Case 2.} The new vertex is added to the edge that connects the vertices $u$ and $v$. Here $u$, $v$ could be $w_i$ or $z_j$.
\begin{figure}[H]
	\centering
	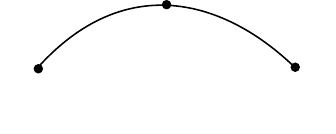
	\caption[]{CASE 2.}
	\label{Fig:Q1}
\end{figure}

By Proposition \ref{prodformula}, we obtain
\begin{align*}
&P(u-w)P(w-v)\\
=~&2\pi i\frac{d}{d\tau}P(u-v)+\zeta(w-u)\frac{d}{du}P(u-v)+P(w-v)P(u-v)\\
 ~&+\zeta(w-v)\frac{d}{dv}P(u-v)+P(w-u)P(u-v).\\
=~&2\pi i\frac{d}{d\tau}P(u-v)\\
  &+\bigg(\wp_1(w-u)-(w-u)\frac{\pi^2}{3}E_2\bigg)\frac{d}{du}P(u-v)+\bigg(\wp_1(w-v)-(w-v)\frac{\pi^2}{3}E_2\bigg)\frac{d}{dv}P(u-v)\\
  &+\bigg(\wp_2(w-u)+\frac{\pi^2}{3}E_2\bigg)P(u-v)+\bigg(\wp_2(w-v)+\frac{\pi^2}{3}E_2\bigg)P(u-v).
\end{align*}

\noindent\textbf{Case 3.} The new vertex is added in the self-loop based on the vertex $w_i$:
\bigskip
\begin{figure}[H]
	\centering
	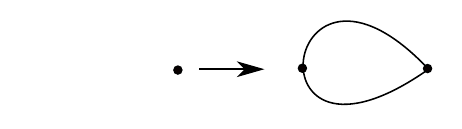
	\caption[]{CASE 3.}
	\label{Fig:Q1}
\end{figure}

In this case, by Proposition \ref{prodformula}, we obtain
$$P(w-w_i)P(w-w_i)=2\pi i\frac{d}{d\tau}\left(\frac{\pi^2}{3}E_2\right)+P(w-w_i)\frac{2\pi^2}{3}E_2+\wp_4(w-w_i).$$

Note that in our definition of the correlation function, there is a scaling factor $\eta(\tau)^{-1}$.
Taking into account the automorphism of the graph and combining Cases 1,2,3 with the following differential equation for the Dedekind eta function $\eta(\tau)$:
$$\frac{d}{d\tau}\log(\eta(\tau))=\frac{\pi i}{12}E_2(\tau),$$
we get the desired equation.
\end{proof}

\end{appendices}

\end{document}

%% file: EX.pdf_tex
\begingroup%
  \makeatletter%
  \providecommand\color[2][]{%
    \errmessage{(Inkscape) Color is used for the text in Inkscape, but the package 'color.sty' is not loaded}%
    \renewcommand\color[2][]{}%
  }%
  \providecommand\transparent[1]{%
    \errmessage{(Inkscape) Transparency is used (non-zero) for the text in Inkscape, but the package 'transparent.sty' is not loaded}%
    \renewcommand\transparent[1]{}%
  }%
  \providecommand\rotatebox[2]{#2}%
  \newcommand*\fsize{\dimexpr\f@size pt\relax}%
  \newcommand*\lineheight[1]{\fontsize{\fsize}{#1\fsize}\selectfont}%
  \ifx\svgwidth\undefined%
    \setlength{\unitlength}{225.13681006bp}%
    \ifx\svgscale\undefined%
      \relax%
    \else%
      \setlength{\unitlength}{\unitlength * \real{\svgscale}}%
    \fi%
  \else%
    \setlength{\unitlength}{\svgwidth}%
  \fi%
  \global\let\svgwidth\undefined%
  \global\let\svgscale\undefined%
  \makeatother%
  \begin{picture}(1,0.22078037)%
    \lineheight{1}%
    \setlength\tabcolsep{0pt}%
    \put(0,0){\includegraphics[width=\unitlength,page=1]{EX.pdf}}%
    \put(0.05652744,0.09006132){\makebox(0,0)[lt]{\lineheight{1.25}\smash{\begin{tabular}[t]{l}$z_1$\end{tabular}}}}%
    \put(0.34808813,0.09006132){\makebox(0,0)[lt]{\lineheight{1.25}\smash{\begin{tabular}[t]{l}$z_2$\end{tabular}}}}%
    \put(0.60526605,0.09206313){\makebox(0,0)[lt]{\lineheight{1.25}\smash{\begin{tabular}[t]{l}$z_1$\end{tabular}}}}%
    \put(0.96086975,0.09307871){\makebox(0,0)[lt]{\lineheight{1.25}\smash{\begin{tabular}[t]{l}$z_2$\end{tabular}}}}%
    \put(0,0){\includegraphics[width=\unitlength,page=2]{EX.pdf}}%
    \put(0.00760255,0.20784253){\makebox(0,0)[lt]{\lineheight{1.25}\smash{\begin{tabular}[t]{l}$Q(z_1,z_1)$\end{tabular}}}}%
    \put(0.34567489,0.20775168){\makebox(0,0)[lt]{\lineheight{1.25}\smash{\begin{tabular}[t]{l}$Q(z_2,z_2)$\end{tabular}}}}%
    \put(0.72411919,0.20454633){\makebox(0,0)[lt]{\lineheight{1.25}\smash{\begin{tabular}[t]{l}$P(z_1,z_2)$\end{tabular}}}}%
    \put(0.72587779,0.00338217){\makebox(0,0)[lt]{\lineheight{1.25}\smash{\begin{tabular}[t]{l}$P(z_1,z_2)$\end{tabular}}}}%
    \put(0.53158656,0.10286259){\makebox(0,0)[lt]{\lineheight{1.25}\smash{\begin{tabular}[t]{l}$+$\end{tabular}}}}%
  \end{picture}%
\endgroup%

%% file: quiver1.pdf_tex
\begingroup%
  \makeatletter%
  \providecommand\color[2][]{%
    \errmessage{(Inkscape) Color is used for the text in Inkscape, but the package 'color.sty' is not loaded}%
    \renewcommand\color[2][]{}%
  }%
  \providecommand\transparent[1]{%
    \errmessage{(Inkscape) Transparency is used (non-zero) for the text in Inkscape, but the package 'transparent.sty' is not loaded}%
    \renewcommand\transparent[1]{}%
  }%
  \providecommand\rotatebox[2]{#2}%
  \newcommand*\fsize{\dimexpr\f@size pt\relax}%
  \newcommand*\lineheight[1]{\fontsize{\fsize}{#1\fsize}\selectfont}%
  \ifx\svgwidth\undefined%
    \setlength{\unitlength}{232.94229126bp}%
    \ifx\svgscale\undefined%
      \relax%
    \else%
      \setlength{\unitlength}{\unitlength * \real{\svgscale}}%
    \fi%
  \else%
    \setlength{\unitlength}{\svgwidth}%
  \fi%
  \global\let\svgwidth\undefined%
  \global\let\svgscale\undefined%
  \makeatother%
  \begin{picture}(1,0.42965781)%
    \lineheight{1}%
    \setlength\tabcolsep{0pt}%
    \put(0,0){\includegraphics[width=\unitlength,page=1]{quiver1.pdf}}%
    \put(0.41891471,0.30563626){\makebox(0,0)[lt]{\lineheight{1.25}\smash{\begin{tabular}[t]{l}$P(z_1-z_2)$\end{tabular}}}}%
    \put(0.41973626,0.03802643){\makebox(0,0)[lt]{\lineheight{1.25}\smash{\begin{tabular}[t]{l}$P(z_1-z_2)$\end{tabular}}}}%
    \put(-0.00129961,0.23658125){\makebox(0,0)[lt]{\lineheight{1.25}\smash{\begin{tabular}[t]{l}$Q(z_1,z_1)$\end{tabular}}}}%
    \put(0.00334314,0.00593514){\makebox(0,0)[lt]{\lineheight{1.25}\smash{\begin{tabular}[t]{l}$Q(z_1,z_1)$\end{tabular}}}}%
    \put(0.88228888,0.07496672){\makebox(0,0)[lt]{\lineheight{1.25}\smash{\begin{tabular}[t]{l}$Q(z_2,z_2)$\end{tabular}}}}%
    \put(0.21209995,0.10937012){\makebox(0,0)[lt]{\lineheight{1.25}\smash{\begin{tabular}[t]{l}$z_1$\end{tabular}}}}%
    \put(0.78061018,0.12287929){\makebox(0,0)[lt]{\lineheight{1.25}\smash{\begin{tabular}[t]{l}$z_2$\end{tabular}}}}%
    \put(0.48022013,0.1744559){\makebox(0,0)[lt]{\lineheight{1.25}\smash{\begin{tabular}[t]{l}$\vdots$\end{tabular}}}}%
  \end{picture}%
\endgroup%

%% file: quiver2.pdf_tex
\begingroup%
  \makeatletter%
  \providecommand\color[2][]{%
    \errmessage{(Inkscape) Color is used for the text in Inkscape, but the package 'color.sty' is not loaded}%
    \renewcommand\color[2][]{}%
  }%
  \providecommand\transparent[1]{%
    \errmessage{(Inkscape) Transparency is used (non-zero) for the text in Inkscape, but the package 'transparent.sty' is not loaded}%
    \renewcommand\transparent[1]{}%
  }%
  \providecommand\rotatebox[2]{#2}%
  \newcommand*\fsize{\dimexpr\f@size pt\relax}%
  \newcommand*\lineheight[1]{\fontsize{\fsize}{#1\fsize}\selectfont}%
  \ifx\svgwidth\undefined%
    \setlength{\unitlength}{85.82607969bp}%
    \ifx\svgscale\undefined%
      \relax%
    \else%
      \setlength{\unitlength}{\unitlength * \real{\svgscale}}%
    \fi%
  \else%
    \setlength{\unitlength}{\svgwidth}%
  \fi%
  \global\let\svgwidth\undefined%
  \global\let\svgscale\undefined%
  \makeatother%
  \begin{picture}(1,0.35011409)%
    \lineheight{1}%
    \setlength\tabcolsep{0pt}%
    \put(0,0){\includegraphics[width=\unitlength,page=1]{quiver2.pdf}}%
    \put(0.25809339,0.32897537){\makebox(0,0)[lt]{\lineheight{1.25}\smash{\begin{tabular}[t]{l}$Q(z_1,z_2)$\end{tabular}}}}%
    \put(0.00808585,0.00484206){\makebox(0,0)[lt]{\lineheight{1.25}\smash{\begin{tabular}[t]{l}$z_1$\end{tabular}}}}%
    \put(0.91753752,0.00897557){\makebox(0,0)[lt]{\lineheight{1.25}\smash{\begin{tabular}[t]{l}$z_2$\end{tabular}}}}%
    \put(0,0){\includegraphics[width=\unitlength,page=2]{quiver2.pdf}}%
  \end{picture}%
\endgroup%

%% file: quiver3.pdf_tex
\begingroup%
  \makeatletter%
  \providecommand\color[2][]{%
    \errmessage{(Inkscape) Color is used for the text in Inkscape, but the package 'color.sty' is not loaded}%
    \renewcommand\color[2][]{}%
  }%
  \providecommand\transparent[1]{%
    \errmessage{(Inkscape) Transparency is used (non-zero) for the text in Inkscape, but the package 'transparent.sty' is not loaded}%
    \renewcommand\transparent[1]{}%
  }%
  \providecommand\rotatebox[2]{#2}%
  \newcommand*\fsize{\dimexpr\f@size pt\relax}%
  \newcommand*\lineheight[1]{\fontsize{\fsize}{#1\fsize}\selectfont}%
  \ifx\svgwidth\undefined%
    \setlength{\unitlength}{60.62814038bp}%
    \ifx\svgscale\undefined%
      \relax%
    \else%
      \setlength{\unitlength}{\unitlength * \real{\svgscale}}%
    \fi%
  \else%
    \setlength{\unitlength}{\svgwidth}%
  \fi%
  \global\let\svgwidth\undefined%
  \global\let\svgscale\undefined%
  \makeatother%
  \begin{picture}(1,0.45600647)%
    \lineheight{1}%
    \setlength\tabcolsep{0pt}%
    \put(0,0){\includegraphics[width=\unitlength,page=1]{quiver3.pdf}}%
    \put(-0.60499329,-0.00523594){\makebox(0,0)[lt]{\lineheight{1.25}\smash{\begin{tabular}[t]{l}$Q(z_1,z_1)$\end{tabular}}}}%
    \put(0.80194339,0.01162954){\makebox(0,0)[lt]{\lineheight{1.25}\smash{\begin{tabular}[t]{l}$z_1$\end{tabular}}}}%
  \end{picture}%
\endgroup%

%% file: quiver4.pdf_tex
\begingroup%
  \makeatletter%
  \providecommand\color[2][]{%
    \errmessage{(Inkscape) Color is used for the text in Inkscape, but the package 'color.sty' is not loaded}%
    \renewcommand\color[2][]{}%
  }%
  \providecommand\transparent[1]{%
    \errmessage{(Inkscape) Transparency is used (non-zero) for the text in Inkscape, but the package 'transparent.sty' is not loaded}%
    \renewcommand\transparent[1]{}%
  }%
  \providecommand\rotatebox[2]{#2}%
  \newcommand*\fsize{\dimexpr\f@size pt\relax}%
  \newcommand*\lineheight[1]{\fontsize{\fsize}{#1\fsize}\selectfont}%
  \ifx\svgwidth\undefined%
    \setlength{\unitlength}{114.16369341bp}%
    \ifx\svgscale\undefined%
      \relax%
    \else%
      \setlength{\unitlength}{\unitlength * \real{\svgscale}}%
    \fi%
  \else%
    \setlength{\unitlength}{\svgwidth}%
  \fi%
  \global\let\svgwidth\undefined%
  \global\let\svgscale\undefined%
  \makeatother%
  \begin{picture}(1,0.24420712)%
    \lineheight{1}%
    \setlength\tabcolsep{0pt}%
    \put(0,0){\includegraphics[width=\unitlength,page=1]{quiver4.pdf}}%
    \put(-0.00265176,0.00973783){\makebox(0,0)[lt]{\lineheight{1.25}\smash{\begin{tabular}[t]{l}$z_1$\end{tabular}}}}%
    \put(0.9075821,0.00617602){\makebox(0,0)[lt]{\lineheight{1.25}\smash{\begin{tabular}[t]{l}$z_i$\end{tabular}}}}%
    \put(0.32001957,0.21724476){\makebox(0,0)[lt]{\lineheight{1.25}\smash{\begin{tabular}[t]{l}$P(z_1-z_i)$\end{tabular}}}}%
    \put(0,0){\includegraphics[width=\unitlength,page=2]{quiver4.pdf}}%
  \end{picture}%
\endgroup%

%% file: case1.pdf_tex
\begingroup%
  \makeatletter%
  \providecommand\color[2][]{%
    \errmessage{(Inkscape) Color is used for the text in Inkscape, but the package 'color.sty' is not loaded}%
    \renewcommand\color[2][]{}%
  }%
  \providecommand\transparent[1]{%
    \errmessage{(Inkscape) Transparency is used (non-zero) for the text in Inkscape, but the package 'transparent.sty' is not loaded}%
    \renewcommand\transparent[1]{}%
  }%
  \providecommand\rotatebox[2]{#2}%
  \newcommand*\fsize{\dimexpr\f@size pt\relax}%
  \newcommand*\lineheight[1]{\fontsize{\fsize}{#1\fsize}\selectfont}%
  \ifx\svgwidth\undefined%
    \setlength{\unitlength}{128.69305828bp}%
    \ifx\svgscale\undefined%
      \relax%
    \else%
      \setlength{\unitlength}{\unitlength * \real{\svgscale}}%
    \fi%
  \else%
    \setlength{\unitlength}{\svgwidth}%
  \fi%
  \global\let\svgwidth\undefined%
  \global\let\svgscale\undefined%
  \makeatother%
  \begin{picture}(1,0.28613189)%
    \lineheight{1}%
    \setlength\tabcolsep{0pt}%
    \put(0,0){\includegraphics[width=\unitlength,page=1]{case1.pdf}}%
    \put(0.82729754,0.12355392){\makebox(0,0)[lt]{\lineheight{1.25}\smash{\begin{tabular}[t]{l}$w$\end{tabular}}}}%
    \put(-0.00235237,0.11843731){\makebox(0,0)[lt]{\lineheight{1.25}\smash{\begin{tabular}[t]{l}$\frac{\pi^2}{3}E_2$\end{tabular}}}}%
    \put(0,0){\includegraphics[width=\unitlength,page=2]{case1.pdf}}%
  \end{picture}%
\endgroup%

%% file: case2.pdf_tex
\begingroup%
  \makeatletter%
  \providecommand\color[2][]{%
    \errmessage{(Inkscape) Color is used for the text in Inkscape, but the package 'color.sty' is not loaded}%
    \renewcommand\color[2][]{}%
  }%
  \providecommand\transparent[1]{%
    \errmessage{(Inkscape) Transparency is used (non-zero) for the text in Inkscape, but the package 'transparent.sty' is not loaded}%
    \renewcommand\transparent[1]{}%
  }%
  \providecommand\rotatebox[2]{#2}%
  \newcommand*\fsize{\dimexpr\f@size pt\relax}%
  \newcommand*\lineheight[1]{\fontsize{\fsize}{#1\fsize}\selectfont}%
  \ifx\svgwidth\undefined%
    \setlength{\unitlength}{161.12811592bp}%
    \ifx\svgscale\undefined%
      \relax%
    \else%
      \setlength{\unitlength}{\unitlength * \real{\svgscale}}%
    \fi%
  \else%
    \setlength{\unitlength}{\svgwidth}%
  \fi%
  \global\let\svgwidth\undefined%
  \global\let\svgscale\undefined%
  \makeatother%
  \begin{picture}(1,0.39686444)%
    \lineheight{1}%
    \setlength\tabcolsep{0pt}%
    \put(0,0){\includegraphics[width=\unitlength,page=1]{case2.pdf}}%
    \put(0.46112487,0.32004944){\makebox(0,0)[lt]{\lineheight{1.25}\smash{\begin{tabular}[t]{l}$w$\end{tabular}}}}%
    \put(-0.00281827,0.17400458){\makebox(0,0)[lt]{\lineheight{1.25}\smash{\begin{tabular}[t]{l}$u$\end{tabular}}}}%
    \put(0.93319796,0.17224502){\makebox(0,0)[lt]{\lineheight{1.25}\smash{\begin{tabular}[t]{l}$v$\end{tabular}}}}%
    \put(0.02647193,0.37148491){\makebox(0,0)[lt]{\lineheight{1.25}\smash{\begin{tabular}[t]{l}$P(u-w)$\end{tabular}}}}%
    \put(0.75102573,0.37147492){\makebox(0,0)[lt]{\lineheight{1.25}\smash{\begin{tabular}[t]{l}$P(w-v)$\end{tabular}}}}%
    \put(0,0){\includegraphics[width=\unitlength,page=2]{case2.pdf}}%
    \put(0.37107792,0.00499409){\makebox(0,0)[lt]{\lineheight{1.25}\smash{\begin{tabular}[t]{l}\textcolor{red}{$P(u-v)$}\end{tabular}}}}%
  \end{picture}%
\endgroup%

%% file: case3.pdf_tex
\begingroup%
  \makeatletter%
  \providecommand\color[2][]{%
    \errmessage{(Inkscape) Color is used for the text in Inkscape, but the package 'color.sty' is not loaded}%
    \renewcommand\color[2][]{}%
  }%
  \providecommand\transparent[1]{%
    \errmessage{(Inkscape) Transparency is used (non-zero) for the text in Inkscape, but the package 'transparent.sty' is not loaded}%
    \renewcommand\transparent[1]{}%
  }%
  \providecommand\rotatebox[2]{#2}%
  \newcommand*\fsize{\dimexpr\f@size pt\relax}%
  \newcommand*\lineheight[1]{\fontsize{\fsize}{#1\fsize}\selectfont}%
  \ifx\svgwidth\undefined%
    \setlength{\unitlength}{221.57241701bp}%
    \ifx\svgscale\undefined%
      \relax%
    \else%
      \setlength{\unitlength}{\unitlength * \real{\svgscale}}%
    \fi%
  \else%
    \setlength{\unitlength}{\svgwidth}%
  \fi%
  \global\let\svgwidth\undefined%
  \global\let\svgscale\undefined%
  \makeatother%
  \begin{picture}(1,0.27867813)%
    \lineheight{1}%
    \setlength\tabcolsep{0pt}%
    \put(0,0){\includegraphics[width=\unitlength,page=1]{case3.pdf}}%
    \put(-0.0013663,0.122971){\makebox(0,0)[lt]{\lineheight{1.25}\smash{\begin{tabular}[t]{l}$\frac{\pi^2}{3}E_2$\end{tabular}}}}%
    \put(0.29303674,0.12242974){\makebox(0,0)[lt]{\lineheight{1.25}\smash{\begin{tabular}[t]{l}$w_i$\end{tabular}}}}%
    \put(0.94709336,0.12532371){\makebox(0,0)[lt]{\lineheight{1.25}\smash{\begin{tabular}[t]{l}$w_i$\end{tabular}}}}%
    \put(0.60412161,0.11292983){\makebox(0,0)[lt]{\lineheight{1.25}\smash{\begin{tabular}[t]{l}$w$\end{tabular}}}}%
    \put(0.72879859,0.26478595){\makebox(0,0)[lt]{\lineheight{1.25}\smash{\begin{tabular}[t]{l}$P(w-w_i)$\end{tabular}}}}%
    \put(0.7403995,0.00275373){\makebox(0,0)[lt]{\lineheight{1.25}\smash{\begin{tabular}[t]{l}$P(w-w_i)$\end{tabular}}}}%
    \put(0,0){\includegraphics[width=\unitlength,page=2]{case3.pdf}}%
  \end{picture}%
\endgroup%